\documentclass[12pt,reqno]{amsart}

\usepackage{amssymb}
\usepackage{amscd}
\usepackage{amsfonts}
\usepackage{setspace}
\usepackage{version}
\usepackage{mathrsfs}
\usepackage{mathtools} 
\usepackage{enumitem}
\usepackage{xcolor}
\usepackage{hyperref}
\hypersetup{
  colorlinks,
  linkcolor={blue!50!black},
  citecolor={green!40!black},
  urlcolor={blue!80!black}
}
\usepackage{breakurl}

\usepackage{graphicx}

\newtheorem{theorem}{Theorem}[section]
\newtheorem{lemma}[theorem]{Lemma}
\newtheorem{proposition}[theorem]{Proposition}
\newtheorem{corollary}[theorem]{Corollary}
\newtheorem{conjecture}[theorem]{Conjecture}

\theoremstyle{definition}

\theoremstyle{remark}
\newtheorem*{remark*}{Remark}

\usepackage{todonotes}

\renewcommand{\leq}{\leqslant}
\renewcommand{\geq}{\geqslant}

\newcommand\R{\mathbf{R}}

\newcommand\Z{\mathbf{Z}}
\newcommand\ZZ{\Z}

\newcommand\F{\mathbf{F}}
\renewcommand\P{\mathbf{P}}

\newcommand\KL{{\operatorname{KL}}}

\newcommand\N{\mathbf{N}}

\newcommand\ent{{\operatorname{ent}}}

\renewcommand\H{\mathbf{H}}
\newcommand\h{\mathrm{h}}
\newcommand\eps{\varepsilon}

\newcommand\PFR[0]{\operatorname{PFR}}

\newcommand\CC{{C_1}}

\numberwithin{equation}{section}

\begin{document}

\title[Sumsets and entropy revisited]{Sumsets and entropy revisited}
\author{Ben Green}
\thanks{BG is supported by Simons Investigator Award 376201.}
\address{Mathematical Institute \\ Andrew Wiles Building \\ Radcliffe Observatory Quarter \\ Woodstock Rd \\ Oxford OX2 6QW \\ UK}
\email{ben.green@maths.ox.ac.uk}

\author{Freddie Manners}
\thanks{FM is supported by a Sloan Fellowship.  During part of the preparation of this work, he was also supported by a von Neumann Fellowship from the Institute for Advanced Study.}
\address{Department of Mathematics \\ University of California, San Diego (UCSD)\\
9500 Gilman Drive \# 0112 \\ La Jolla, CA  92093-0112 \\ USA} 
\email{fmanners@ucsd.edu}

\author{Terence Tao}
\thanks{TT is supported by NSF grant DMS-1764034 and by a Simons Investigator Award.  }
\address{Math Sciences Building \\ 520 Portola Plaza \\ Box 951555 \\ Los Angeles, CA 90095 \\ USA}
\email{tao@math.ucla.edu}



\subjclass[2000]{Primary 11B30, 94A17}

\begin{abstract}
The entropic doubling $\sigma_{\ent}[X]$ of a random variable $X$ taking values in an abelian group $G$ is a variant of the notion of the doubling constant $\sigma[A]$ of a finite subset $A$ of $G$, but it enjoys somewhat better properties; for instance, it contracts upon applying a homomorphism.

In this paper we develop further the theory of entropic doubling and give various applications, including:
\begin{enumerate}
\item A new proof of a result of P\'alv\"olgyi and Zhelezov on the ``skew dimension'' of subsets of $\Z^D$ with small doubling;
\item A new proof, and an improvement, of a result of the second author on the dimension of subsets of $\Z^D$ with small doubling;
\item A proof that the Polynomial Freiman--Ruzsa conjecture over $\F_2$ implies the (weak) Polynomial Freiman--Ruzsa conjecture over $\Z$.
\end{enumerate}
\end{abstract}
\maketitle

\tableofcontents

\section{Introduction and statement of results}

\emph{Notation.} Throughout the paper we use standard asymptotic notation. The notations $X = O(Y)$, $X \ll Y$, or $Y \gg X$ all denote the bound $|X| \leq CY$ for an absolute constant $C$. Different instances of the notation may imply different constants $C$.

\subsection{Entropy doubling and Ruzsa distance}

Let $G = (G,+)$ be an abelian group. In this paper, by a ``$G$-valued random variable'' we mean a random variable $X$ taking values in a finite subset of $G$.  Given such a variable, the \emph{entropic doubling constant} (first introduced in~\cite{tao-entropy}) $\sigma_\ent[X]$ is defined by the formula
\[ \sigma_{\ent}[X] \coloneqq  \exp \bigl(\H(X_1 + X_2) - \H(X)\bigr),\] 
where $X_1, X_2$ are independent copies of $X$. Here, $\H(X)$ denotes the Shannon entropy of $X$, the definition and basic properties of which we review in Appendix~\ref{basic-entropy-facts}. 

If $A \subseteq G$ is a finite non-empty set, by abuse of notation we write $\sigma_{\ent}[A] = \sigma_{\ent}[U_A]$, where $U_A$ is a uniform random variable drawn from $A$.  For instance, if $H$ is a finite subgroup of $G$, one can check that $\sigma_\ent[H]=1$.  

The entropic doubling constant is related to other standard measures of additive structure via the inequalities
\begin{equation}%
  \label{structure-ineqs}
  \frac{|A|^3}{\mathrm{E}[A]} \leq \sigma_{\ent}[A] \leq \sigma[A].
\end{equation}
Here, $\sigma[A] \coloneqq \frac{|A+A|}{|A|}$ is the doubling constant of $A$ and \[ \mathrm{E}[A] \coloneqq |\{ (a_1, a_2, a_3, a_4) \in A^4 : a_1 + a_2 = a_3 + a_4\}| \] is the additive energy of $A$.

The second inequality in~\eqref{structure-ineqs} was noted in~\cite[Equation 10]{tao-entropy}, and we recall the proof in Appendix~\ref{app:energy-entropy}. The first inequality seems not to have appeared explicitly in the literature, but it follows easily either by a direct argument using the weighted AM--GM inequality, or by quoting the monotonicity of R\'enyi entropy. For completeness, we give this argument in Appendix~\ref{app:energy-entropy}. Both inequalities can be far from tight: we give an example in Appendix~\ref{app:energy-entropy}.

\subsubsection*{Entropic Ruzsa distance}

If $X, Y$ are $G$-valued random variables (not necessarily independent, or even defined on the same sample space), then we define the \emph{entropic Ruzsa distance} $d_\ent(X,Y)$ between these variables by the formula
\begin{equation}\label{dent-def} d_{\ent}(X,Y) \coloneqq  \H(X' - Y') - \frac{1}{2} \H(X') - \frac{1}{2} \H(Y'),\end{equation} where $X', Y'$ are independent copies of $X, Y$ respectively. 
This concept, introduced by Ruzsa~\cite{ruzsa-entropy} and studied in more detail by the third author~\cite{tao-entropy},  generalizes entropic doubling, since $\sigma_{\ent}[X] = e^{d_{\ent}(X, -X)}$.  

It is easy to see that $d_{\ent}(X, Y) = d_\ent(Y,X) \geq 0$, and also that $d_\ent(U_H,U_H)=0$ for any finite subgroup $H$ of $G$. Note that $d_\ent(X,Y)$ depends only on the distributions 
\[p_X(x) \coloneqq \P(X=x); \quad p_Y(y) \coloneqq \P(Y=y)\]
of $X,Y$. We have (see the final paragraph of~\cite{ruzsa-entropy},~\cite[Theorem 1.10]{tao-entropy}, or Lemma~\ref{improv-ruzsa} below) the \emph{entropic Ruzsa triangle inequality}
\begin{equation}\label{triangle} d_{\ent}(X,Z) \leq d_{\ent}(X, Y) + d_{\ent}(Y, Z)\end{equation} 
for any three $G$-valued random variables $X, Y, Z$.  

\begin{remark*}
It is somewhat traditional to use the letter $K$ for the combinatorial doubling constant $\sigma[A]$. We will generally use the letter $k$ for distances $d_{\ent}(X, Y)$. Where these arise from sets (for instance if $X = Y = U_A$) one should informally think of $k$ being on the order of $\log K$. It should be carefully noted that $k$ may take values in $[0, \infty)$ and is not constrained to be an integer. 
\end{remark*}

It will be technically convenient to introduce a small modification of the entropic Ruzsa distance.  Define the \emph{maximal entropic Ruzsa distance} $d_\ent^*(X,Y)$ to be the quantity
\begin{equation}\label{max-dist}
  d^*_{\ent}(X,Y) \coloneqq \sup_{X',Y'} \left( \H(X' - Y') - \frac{1}{2} \H(X') - \frac{1}{2} \H(Y') \right)
\end{equation}
where $X',Y'$ range over all pairs of random variables with marginal distributions $p_X$, $p_Y$ respectively (i.e., all couplings of $X$ and $Y$). In particular, $X', Y'$ are \emph{not} required to be independent.  

We have the following observations.
\begin{lemma}\label{improv-ruzsa} Let $X,Y,Z$ be $G$-valued random variables. Then:
  \begin{enumerate}[label={\textup{(\roman*)}}]
    \item We have $d^*_\ent(X,Z) \leq d_\ent(X,Y) + d_\ent(Y,Z)$.
    \item We have $d_\ent(X,Y) \leq d^*_\ent(X,Y) \leq 3 d_\ent(X,Y)$.
  \end{enumerate}
\end{lemma}
For the proof, see Section~\ref{improv-sec}.

\subsubsection*{Small Ruzsa distance}%
We turn now to our first main result, which gives a closer connection than~\eqref{structure-ineqs} between small entropy doubling (or more generally small Ruzsa distance) and small doubling.
Here, for a real parameter $p \in (0,1)$, we write $\h(p) \coloneqq  p \log \frac{1}{p} + (1 - p) \log \frac{1}{1-p}$ for the entropy of the Bernouilli random variable with probability $p$.

\begin{proposition}%
  \label{sec2-prop}
  Let $C \geq 4$ be a real parameter. For any $G$-valued random variables $X,Y$ there is a non-empty finite subset $S$ of $G$ such that, if $U_S$ is a uniform random variable on $S$, then 
  \begin{equation}%
    \label{approx-xs}
    d_{\ent}(U_S, Y) \leq (C + 2) d_\ent(X,Y) + \h\biggl(1 - \frac{2}{C}\biggr)
  \end{equation}
  and
  \begin{equation}%
    \label{s-small-doubling}
    \log \frac{ |S - S|}{|S|} \leq (2C + 4) d_\ent(X,Y)  + 2 \h\biggl(1 - \frac{2}{C}\biggr) .
  \end{equation}
\end{proposition}

We isolate two special cases of this proposition for future use:
\begin{enumerate}[label=(\roman*)]
  \item If we take $C=4$ in the above proposition, then (using, in addition, Lemma~\ref{improv-ruzsa} (ii)) we obtain the bounds $d_{\ent}(U_S, Y) \leq 6d_\ent(X,Y)  +  \log 2$ and $|S - S| \leq 4 e^{12 d_\ent(X,Y)} |S|$.
  \item If $d_\ent(X,Y) = \eps$ for some $0 < \eps \leq \frac{1}{16}$, then on taking $C \coloneqq \eps^{-1/2}$ we obtain the bounds $d_\ent(U_S,Y) \ll \eps^{1/2} \log \frac{1}{\eps}$ and $|S-S| \leq \bigl(1 + O( \eps^{1/2} \log \frac{1}{\eps} )\bigr) |S|$.
\end{enumerate}
We also remark that a qualitative version of this proposition (with unspecified dependence on $d_\ent(X,Y)$ on the right-hand side) was previously established in~\cite[Proposition 5.2]{tao-entropy}. 

In the regime where $d_\ent(X,Y)$ is small, we can in fact obtain the following more precise result.
\begin{proposition}\label{sec3-prop}
  There is absolute constant $\eps_0>0$ such that the following is true.
Let $X, Y$ be $G$-valued random variables, and suppose that $d_\ent(X,Y) \leq \eps_0$. Then there is some finite subgroup $H$ of $G$ such that $d_\ent(X, U_H), d_{\ent}(Y, U_H) \leq 12 d_\ent(X,Y)$.
\end{proposition}

\begin{remark*}
  The constant $\eps_0$ could be specified explicitly if desired, but we have not carried out such a calculation. The constant $12$ can be improved, but we will not attempt to optimise it here.
\end{remark*}

\subsubsection*{Behaviour under homomorphisms}
Given the close relation between notions of entropy doubling and Ruzsa distance and the usual combinatorial notions, one would be forgiven for wondering what the point of introducing the former is. 

The answer is that the entropy notions are more flexible and behave better in various ways. Most obviously, they are defined for arbitrary random variables $X$, with no requirement that $X$ be uniform on a set. Related to this is the fact that the entropy notions behave well under homomorphisms in a way that the combinatorial notions do not. 

The following is our main result in this direction. Here (and below) we use $(X|E)$ to denote a random variable $X$ conditioned to a positive probability event $E$. We adopt the convention that expressions such as $p_{Y_1}(y_1) p_{Y_2}(y_2) d_\ent\bigl((X_1 | Y_1 = y_1), (X_2 | Y_2 = y_2)\bigr)$ vanish if one of the events $Y_1=y_1$, $Y_2=y_2$ occurs with zero probability.

\begin{proposition}\label{projections-1}  Let $\pi \colon G \to H$ be a homomorphism, let $X_1,X_2$ be $G$-valued random variables, and set $Y_1 \coloneqq \pi(X_1)$ and $Y_2 \coloneqq \pi(X_2)$.
Then we have
\begin{align*}
  d_\ent(X_1, X_2) \geq & \; d_\ent (Y_1, Y_2) \\ & + \sum_{y_1, y_2 \in H}  p_{Y_1}(y_1) p_{Y_2}(y_2) d_\ent\bigl((X_1 | Y_1 = y_1), (X_2 | Y_2 = y_2)\bigr).
\end{align*}
In particular, we have
\[
  d_\ent(X_1, X_2) \geq d_\ent(Y_1, Y_2)
\]
and thus
\begin{equation}\label{doubling-compress}
 \sigma_\ent[\pi(X)] \leq \sigma_\ent[X]
 \end{equation}
for any $G$-valued random variable $X$.
\end{proposition}

\begin{remark*} The main inequality here is a precise version of the intuition that the doubling constant of a subset of $G$ in the presence of a homomorphism $\pi : G \rightarrow H$ should somehow be at least the doubling constant of the `base' times some combination of the doubling constants of the `fibres'. To make sense of this rigorously we need to pass from sets to general random variables, and replace combinatorial doubling by entropy doubling.

  An example of the failure of a similar result in the purely combinatorial setting (in fact of the analogue of~\eqref{doubling-compress}) is outlined in~\cite[Exercise 2.2.10]{tao-vu}.
\end{remark*}

\begin{remark*}
  Many previous works have noted the advantageous properties of entropy in somewhat related settings.
  To give a few examples, in rough chronological order there is the work of Avez~\cite{avez} and of Ka\u{\i}manovich and Vershik~\cite{kv} on random walks on discrete groups, the work of Hochman~\cite{hochman} on fractals, and the work of Breuillard--Varj\'u~\cite{bv} on Bernouilli convolutions.
\end{remark*}

\subsection{Structure of sets with small doubling.} 

We turn now to applications of the results of the previous subsection to inverse theorems for sets with small doubling.

\subsubsection*{Skew dimension}
The first application is a new proof of a result of P\'alv\"olgyi and Zhelezov~\cite{pz}, which they used to give a new and much shorter proof of a celebrated result of Bourgain and Chang~\cite{bourgain-chang}.
For the purposes of this result, an \emph{affine coordinate space} is a subset of $\Z^D$ (for some $D$) obtained by fixing the values of some possibly empty set of coordinates.
For instance, $\{(2,-1, x_3) : x_3 \in \Z\} \subseteq \Z^3$ is an affine coordinate space.
If $A$ is a finite subset of some affine coordinate space $V$, its\footnote{P\'alv\"olgyi and Zhelezov use the term \emph{query complexity} instead of skew dimension.} \emph{skew-dimension} $\dim_*(A)$ is defined inductively, as follows: 
\begin{itemize}
\item $\dim_*(A) = 0$ if and only if $A$ is a singleton (or empty).
\item If $r \geq 1$, then $\dim_*(A) \leq r$ if and only if there is a coordinate map $\pi \colon V \rightarrow \Z$ such that $\dim_*( \pi^{-1}(n) \cap A ) \leq r-1$ for all $n \in \Z$.
\end{itemize}

\begin{theorem}[P\'alv\"olgyi--Zhelezov]\label{pz-main}
Suppose $A$ is a finite subset of some affine coordinate space with $\sigma[A] \leq K$ for some $K \geq 2$.
Then there is $A' \subseteq A$ with $|A'| \geq K^{-O(1)} |A|$ and $\dim_* A' \ll \log K$.
\end{theorem}

This result is essentially contained in a paper~\cite{pz} of P\'alv\"olgyi and Zhelezov: whilst it is not actually stated in that paper, it is mentioned in a talk by Zhelezov~\cite[minute 27:30]{zhelezov-talk}, and can be established using the methods of~\cite{pz}. Lecture notes of the first author may be consulted for a detailed account with some simplifications~\cite{green-pz}, as well as details of the deduction of the result of Bourgain and Chang.

In fact we will establish the following slightly stronger result.

\begin{theorem}\label{pz-actual}
There is an absolute constant $C$ with the following property. Suppose that $A, B \subseteq \Z^D$. Then there are $A' \subseteq A$ and $B' \subseteq B$ with $|A'||B'| \geq e^{-Cd_\ent(A,B)}|A||B|$, and such that $\dim_* (A'), \dim_*(B') \leq C d_\ent(A, B)$.
\end{theorem}

Setting $B=-A$ and using~\eqref{structure-ineqs}, we recover Theorem~\ref{pz-main}.  The ``bilinear'' form of Theorem~\ref{pz-actual} will be convenient for induction purposes.

\subsubsection*{Dimension and PFR over $\Z$}
Whilst the notion of skew-dimension is useful in the context of the work of Bourgain and Chang, the actual (affine) dimension $\dim A$, defined as the dimension of the span of $A-A$ over the reals, is a more intrinsically natural quantity.  Note that $\dim_* A \leq \dim A$ for any $A$.
It is conjectured that Theorem~\ref{pz-main} remains true with $\dim_* A'$ replaced by $\dim A'$ -- this is sometimes (see, for instance,~\cite{mcc},~\cite[Conjecture 1]{pz}) called the weak Polynomial Freiman--Ruzsa conjecture (PFR) over $\Z$.  (The term `weak' comes from the fact that only the dimension is controlled, and no attempt is made to put $A$ economically inside a box or homomorphic image of the lattice points in a convex set. Such stronger statements are also speculated to be true, but care must be taken in their formulation, as discussed in~\cite{lovett-regev}.)

\begin{conjecture}[Weak PFR over $\Z$]%
  \label{weak-pfr-z}
  Suppose that $A \subseteq \Z^D$ is a set with $\sigma[A] \leq K$. Then there is a subset $A' \subseteq A$, $|A'| \geq K^{-O(1)}|A|$, with $\dim A' \ll \log K$.
\end{conjecture}

Prior to this paper, the best known bound in the direction of this conjecture was a result of the second author~\cite{manners}.

\begin{theorem}[{\cite[Theorem 1.5]{manners}}]%
  \label{manners-thm}
  Suppose that $A \subseteq \ZZ^D$ is
	a finite set with $\sigma[A] \leq K$.
  Then there is $A' \subseteq A$ with $\frac{|A'|}{|A|} \gg \exp(-C \log^2 K)$ and $\dim A' \ll \log K$. 
\end{theorem}

The proof of this result in~\cite{manners} was a little exotic, making use of projections modulo $2$ and a kind of ``$U^3$-energy''. We provide a new, shorter, proof of this result, retaining the first feature but using entropic notions in place of the exotic energy. 

We will eventually go further in this paper by using results on sets with additive structure in $\F_2^D$ to improve the bounds, but before doing that we make a detour into the world of structure theorems for sets of small doubling in $\F_2^D$.

\subsubsection*{Small doubling in $\F_2^D$ and PFR over $\F_2$}
In the following discussion, $\F_2^D$ denotes the vector space of dimension $D$ over $\F_2$; the value of $D$ is typically somewhat unimportant. Essentially everything we have to say would work equally well over other finite fields $\F$, but this introduces some further technicalities and implied constants would need to depend on $\F$, and we do not discuss this aspect here.

Denote by $C_{\PFR}$ any
constant for which the following statement is true: if $A \subseteq \F_2^D$, and if the doubling constant $\sigma[A]$ is at most $K$, then $A$ is covered by $\exp(O(\log^{C_{\PFR}}(2 K)))$ cosets of some subspace $H \leq \F_2^D$ of size at most $|A|$. The implied constant in the $O()$ notation is allowed to depend on $C_{\PFR}$.

A celebrated result of Sanders~\cite[Corollary A.2]{sanders} (together with standard covering lemmas) is that one may take $C_{\PFR} = 4$. By an improved version of the argument due to Konyagin (see~\cite[Theorem 1.4]{sanders2}), one can in fact take any $C_{\PFR} > 3$.  Strictly speaking, the statement in~\cite{sanders2} applies to more general abelian groups than $\F_2^D$, but replaces the subspace $H$ by a convex coset progression.  However, an inspection of the arguments in the characteristic $2$ case shows that the convex coset progression in this case can be taken to be a subspace (basically because the convex coset progressions are constructed via Bohr sets, which are automatically subspaces in the characteristic $2$ setting).  Alternatively, one can invoke the discrete John theorem (see~\cite[Theorem 1.6]{taovu-john}) to control the convex coset progression by a generalized arithmetic progression (up to acceptable losses, and increasing $C_{\PFR}$ by an epsilon), and then observe that in $\F_2^D$, all generalized arithmetic progressions are in fact subspaces.  We leave the details of these arguments to the interested reader.

We have the following notorious conjecture, known as the Polynomial Freiman--Ruzsa conjecture over $\F_2$.

\begin{conjecture}\label{pfr-conj} We may take $C_{\PFR} = 1$. 
\end{conjecture}

There are a large number of equivalent formulations of Conjecture~\ref{pfr-conj}; see~\cite{green-finite-field-notes,green-tao-u3-pfr,lovett-u3-pfr,samorodnitsky}.  We add a further equivalent form of Conjecture~\ref{pfr-conj}, formulated in terms of entropy. It says that, in the case $G = \F_2^D$, Proposition~\ref{sec3-prop} is valid with no smallness restriction on the entropic distance $d_\ent(X,Y)$.

\begin{proposition}%
  \label{pfr-f2-prop}
  Conjecture~\ref{pfr-conj} is equivalent to the claim that, for any $\F_2^D$-valued random variables $X,Y$, there is a finite subgroup $H \leq \F_2^D$ such that $d_{\ent}(X, U_H) \ll d_{\ent}(X, Y)$.
\end{proposition}

\subsubsection*{Small doubling and dimension, again}
We now return to the main topic, and offer the following improvement of Theorem~\ref{manners-thm}.

\begin{theorem}%
  \label{manners-thm-new}
  Suppose that $A \subseteq \Z^D$ is a finite set with $\sigma[A] \leq K$. Then there is $A' \subseteq A$ with $\frac{|A'|}{|A|} \geq \exp(-C \log^{2 - \frac{1}{C_{\PFR}}} K)$ and $\dim A' \ll \log K$.  
\end{theorem}
We remark that the constant $C$ is allowed to depend on $C_{\PFR}$ (and so by implication on the implied constant in the definition of $C_{\PFR}$).
As it turns out, the implied constant in the $\ll$ is independent of $C_{\PFR}$, and in particular can be taken to be $40/\log 2$.

Thus, using the Konyagin/Sanders result we can obtain 
\[ \frac{|A'|}{|A|} \gg \exp(-C\log^{5/3 + o(1)} K),\] which is the strongest unconditional result currently known. Perhaps more interestingly, we obtain the following conditional implication between the two forms of the Polynomial Freiman--Ruzsa conjecture discussed above.

\begin{corollary}%
  \label{pfr-weak-pfr}
	Conjecture~\ref{pfr-conj} implies Conjecture~\ref{weak-pfr-z}.
\end{corollary}

\subsubsection*{Plan of the paper}
We begin by developing the theory of entropy doubling and (entropy) Ruzsa distance, as discussed above. In Section~\ref{improv-sec} we establish Lemma~\ref{improv-ruzsa}.  In Section~\ref{sec2}, we look at the link between random variables and sets and establish Proposition~\ref{sec2-prop}. In Section~\ref{section-3} we look at homomorphisms and give the (short) proof of Proposition~\ref{projections-1}. Then, in Section~\ref{sec-4-very-small}, we combine these results to prove Proposition~\ref{sec3-prop}, which relates random variables with small entropy doubling to subgroups. This section is a little lengthy but, as we indicate at the appropriate points, not all of the analysis is needed in subsequent sections.

After this, we turn to the applications to small doubling in $\Z^D$. We begin, in Section~\ref{sec5}, by proving Theorem~\ref{pz-actual} (and thus reproving Theorem~\ref{pz-main}). Then, we give a new proof of the result of the second author, Theorem~\ref{manners-thm}.

Next, we take a brief detour into structural results over $\F_2$, establishing the equivalence of the Polynomial Freiman--Ruzsa conjecture in this setting with its entropic formulation (Proposition~\ref{pfr-f2-prop}).

Finally, we return to small doubling in $\Z^D$, establishing Theorem~\ref{manners-thm-new} (and hence Corollary~\ref{pfr-weak-pfr}) in Section~\ref{sec4}. 

Finally, we remark that although we have written the paper in the context of an abelian group $G$, many of the arguments (for example the proof of Theorem~\ref{sec3-prop}) do not require this assumption.

\subsubsection*{Acknowledgements} We thank Zachary Hunter and Noah Kravitz for some corrections to the first version of the paper.

\section{An improved entropic Ruzsa triangle inequality}\label{improv-sec}

In this section we prove Lemma~\ref{improv-ruzsa}.

\begin{proof}[Proof of Lemma~\ref{improv-ruzsa}]
We begin with part (i).  It suffices to show that
\[ \H(X - Z) - \frac{1}{2}(\H(X) + \H(Z)) \leq d_{\ent}(X, Y) + d_{\ent}(Y,Z).\]
This is equivalent to establishing
\[ \H(X - Z) \leq \H(X-Y) + \H(Y-Z) - \H(Y)\]
whenever $Y$ is independent of $(X, Z)$ (but $X$ and $Z$ are not required to be independent of each other).

We apply the submodularity inequality~\eqref{submodularity} with $A = X - Y$, $B = Z$, $C = X - Z$. With these choices we have
\[ \H(A, B, C) = \H(X, Y, Z) = \H(X, Z) + \H(Y),\] 
\[  \H(A, C) = \H(X - Y, X - Z) = \H(X - Y, Y - Z)  \leq \H(X - Y) + \H(Y - Z)\] and \[ \H(B,C) = \H(Z, X - Z) = \H(X, Z).\]
In the second display we used~\eqref{union}.
Applying~\eqref{submodularity} gives part (i) of Lemma~\ref{improv-ruzsa}.

For part (ii), the first inequality $d_\ent(X,Y) \leq d^*_\ent(X,Y)$ is trivial.  For the second inequality, we apply (i) and~\eqref{triangle} to conclude that
\begin{align*} d^*_\ent(X,Y) & \leq d_\ent(X,Y) + d_\ent(Y,Y) \\ & \leq d_\ent(X,Y) + d_\ent(Y,X) + d_\ent(X,Y),\end{align*}
giving the claim.
\end{proof}

\section{From random variables to sets}\label{sec2}

The objective of this section is to prove Proposition~\ref{sec2-prop}.  Let $C, X, Y$ be as in that proposition.  We may assume without loss of generality that $X, Y$ are independent.  For brevity we adopt the notation $k \coloneqq d_\ent(X,Y)$. We need to locate a set $S$ satisfying~\eqref{approx-xs} and~\eqref{s-small-doubling}. The key lemma is the following.

\begin{lemma}\label{s-comp} There exists a finite non-empty subset $S$ of $G$ such that
\begin{equation}\label{S-size}
  \log |S| \geq \H(Y) - 2 \h\left(1 - \frac{2}{C}\right) - 4 k
\end{equation}
and such that
\begin{equation}\label{zy-bound}
  d_{\ent}(Z, Y) \leq C k + \frac{1}{2} \bigl(\H(Y) - \H(Z)\bigr)
\end{equation}
whenever $Z$ is an $S$-valued random variable.
\end{lemma}

\begin{proof}  As $X,Y$ are independent and $k = d_\ent(X,Y)$, we have
\begin{equation}\label{xy-est}
 \H(X-Y) = \frac{1}{2} \H(X) + \frac{1}{2} \H(Y) + k,
\end{equation}
and hence by~\eqref{ent-lower} it follows that
\[ \H(X - Y) - \H(Y) \leq 2k.\]
Applying the equality case of~\eqref{hzy}, we conclude that
\begin{equation}\label{xo}
\sum_x p_X(x) D_{\KL}(x - Y \Vert X - Y) \leq 2k.
\end{equation}

Inspired by this, we define $S$ by the formula
\begin{equation}%
  \label{s-def}
  S \coloneqq \bigl\{ x : p_X(x) > 0, D_{\KL}(x - Y \Vert X - Y) \leq C k\bigr\}.
\end{equation}
Denote by $A$ the random variable $A \coloneqq 1_{X \in S}$, and write $p \coloneqq \P(A = 1) = \P(X \in S)$. 
 By Markov's inequality and~\eqref{xo}, it follows that
\begin{equation}\label{bad-s-prob} p = \P(X \in S) \geq 1 - \frac{2}{C} \geq \frac{1}{2} .\end{equation}
Now we make some observations. First,
\begin{align}\nonumber
\H(X) &= \H(X, A) \\ \nonumber
& = \H(X |A) + \H(A) \\ 
& = p \H(X |A = 1) + (1 - p) \H(X | A = 0) + \h(p).\label{1-ent}
\end{align}

Second, since $Y$ is independent of $X$ and $A$, it follows using~\eqref{simple-sumset} that 
\[\H(X - Y | A = i) \geq \H(Y), \H(X | A = i)\]
for $i = 0,1$; therefore
\begin{align}\nonumber
  \H(X - Y) & \geq \H(X - Y | A) \\ \nonumber
           & = p\H(X - Y | A = 1) + (1 - p) \H(X - Y | A = 0) \\
                      & \geq p \H(Y) + \frac{1-p}{2}\bigl(\H(Y) + \H(X | A = 0)\bigr). \label{2-ent}
\end{align}

Combining~\eqref{1-ent},~\eqref{2-ent} with~\eqref{xy-est} we conclude after a short computation that 
\begin{equation}%
  \label{3-ent}
  k \geq \frac{p}{2} \H(Y) - \frac{p}{2}  \H(X | A = 1) - \frac{1}{2} \h(p).
\end{equation}

By~\eqref{convexity-bound}, we have $\H(X |A = 1) \leq \log |S|$. Substituting into~\eqref{3-ent} and rearranging yields
\[ \log |S| \geq \H(Y) - \frac{\h(p)}{p} - \frac{2 k}{p}.\] 
Using~\eqref{bad-s-prob} (and the monotone decreasing nature of $\h(p)$ for $p \geq 1/2$), we obtain~\eqref{S-size}.

Now we prove~\eqref{zy-bound}.
We may assume without loss of generality that $X,Y,Z$ are all independent.
From~\eqref{hzy} (replacing $X$ by $X - Y$ there) we have
\[
  \H(Z - Y) - \H(Y) \leq \sum_z p_Z(z) D_{\KL}(z - Y \Vert X - Y).
\]
  Note here that the Kullback--Leibler divergence is well-defined and finite. Indeed, $Z$ takes values $z \in S$, and hence by the definition of $S$ we have $p_X(z) > 0$ for such $z$. Thus if $p_{z - Y}(t) > 0$ then $p_{X - Y}(t) = \sum_{x} p_X(x) p_{x - Y}(t) \geq p_X(z) p_{z - Y}(t) > 0$.
  
  By definition of $S$, $D_{\KL}(z - Y \Vert X - Y) \leq Ck$ for $z$ in the range of $Z$, and the claim~\eqref{zy-bound} follows.  
\end{proof}

Now we are ready for the proof of Proposition~\ref{sec2-prop} itself. 

\begin{proof}[Proof of Proposition~\ref{sec2-prop}] 
We begin by establishing~\eqref{approx-xs}. Let $S$ be as in Lemma~\ref{s-comp}.  Taking $Z = U_S$ in~\eqref{zy-bound}, we have
\[
  d_{\ent}(U_S, Y) \leq C k + \frac{1}{2} \bigl(\H(Y)  -  \log |S|\bigr).
\]
The required bound~\eqref{approx-xs} then follows from~\eqref{S-size}.

Now we prove~\eqref{s-small-doubling}. Let $(Z, Z')$ be any joint random variable with $Z,Z'$ being $S$-valued. From Lemma~\ref{improv-ruzsa}(i) and~\eqref{zy-bound} we have
\begin{align*}
d^*_\ent(Z,Z') & \leq d_\ent(Z,Y) + d_\ent(Z',Y) \\ & \leq 2 C k + \H(Y) - \frac{1}{2} \H(Z) - \frac{1}{2} \H(Z')
\end{align*}
or equivalently
\begin{equation}%
  \label{hzz}
  \H(Z - Z') \leq 2C k + \H(Y).
\end{equation}
Now we observe that it is possible to choose $Z, Z'$  supported on $S$ so that $Z - Z'$ has the uniform distribution on $S - S$.
To do this, simply take $(Z,Z')$ to have distribution function
\[p_{(Z,Z')}(s_1, s_2) \coloneqq \frac{1}{|S - S| \# \{ (t_1,t_2) \in S: t_1-t_2 = s_1 - s_2 \}}\]
for $s_1,s_2 \in S$.
In this case $\H(Z - Z') = \log |S - S|$.  Using~\eqref{hzz} and~\eqref{S-size},~\eqref{s-small-doubling} follows.
\end{proof}

\begin{remark*}
  The last part of this argument has considerable similarity with~\cite[Section 5]{ruzsa-entropy}.
\end{remark*}

\section{Entropy distance under homomorphisms}\label{section-3}
In this section we establish Proposition~\ref{projections-1}. Let notation be as in the statement of that proposition.
\begin{proof}[Proof of Proposition~\ref{projections-1}]
We have
\begin{align}\nonumber \H(X_1 - & X_2 | Y_1, Y_2) \\ & = \sum_{y_1, y_2 \in H} p_{Y_1}(y_1) p_{Y_2}(y_2) \H(X_1 - X_2 | Y_1 = y_1, Y_2 = y_2),\label{eq11}\end{align}
\begin{align} \nonumber \H(X_1 | Y_1) & = \sum_{y_1 \in H} p_{Y_1}(y_1)  \H(X_1 | Y_1 = y_1) \\ & = \sum_{y_1, y_2 \in H} p_{Y_1}(y_1) p_{Y_2}(y_2) \H(X_1 | Y_1 = y_1),\label{eq12}\end{align}
and similarly
\begin{equation}\label{eq13} \H(X_2 | Y_2) = \sum_{y_1, y_2 \in H} p_{Y_1}(y_1) p_{Y_2}(y_2) \H(X_2 | Y_2 = y_2).\end{equation}
Subtracting half of~\eqref{eq12} and half of~\eqref{eq13} from~\eqref{eq11} gives
\begin{align}\nonumber \H(& X_1 - X_2  | Y_1, Y_2) - \frac{1}{2}\H(X_1|Y_1) - \frac{1}{2} \H(X_2 | Y_2) \\ & = \sum_{y_1, y_2 \in H}  p_{Y_1}(y_1) p_{Y_2}(y_2) d_\ent((X_1 | Y_1 = y_1), (X_2 | Y_2 = y_2)).\label{eq14}\end{align}
Now $X_i$ determines $Y_i$, and so 
\begin{equation}\label{eq15} \H(X_1 | Y_1) = \H(X_1) - \H(Y_1), \quad \H(X_2 | Y_2) = \H(X_2) - \H(Y_2).\end{equation}
Moreover, by~\eqref{submodularity-3},
\begin{align}\nonumber \H(X_1 - X_2 | Y_1, Y_2) &  \leq \H(X_1 - X_2 | Y_1 - Y_2) \\ & = \H(X_1 - X_2) - \H(Y_1 - Y_2)\label{eq16}\end{align} (because $X_1 - X_2$ determines $Y_1 - Y_2$).
Combining~\eqref{eq14},~\eqref{eq15},~\eqref{eq16} gives the result. In fact one sees that the difference between the LHS and the RHS in the proposition is $\H(X_1 - X_2 | Y_1 - Y_2) - \H(X_1 - X_2 | Y_1, Y_2)$.
\end{proof}

\section{Very small entropy doubling}\label{sec-4-very-small}

In this section we prove Proposition~\ref{sec3-prop}, which states that random variables $X, Y$ for which $d_{\ent}(X, Y)$ is small are close to uniform on a subgroup. We first handle the case $X = Y$ (in which case we will establish Proposition~\ref{sec3-prop} with the improved constant of $6$). Assume henceforth that $X$ is a $G$-valued random variable with $d_{\ent}(X, X) = \eps \leq \eps_0$. 

We first observe that a weak version of Proposition~\ref{sec3-prop} (which in fact suffices for many applications) follows quickly from Proposition~\ref{sec2-prop}. As observed in item (ii) after Proposition~\ref{sec2-prop}, we see that there is $S$ such that 
\begin{equation}\label{prelim-1} d_{\ent}(U_S, Y) \ll \eps^{1/2} \log \frac{1}{\eps}\end{equation} 
and 
\begin{equation}\label{prelim-2} |S-S| \leq \biggl( 1 + O\biggl( \eps^{1/2} \log \frac{1}{\eps} \biggr) \biggr) |S| < \frac{3}{2} |S|,\end{equation}
where the last inequality holds if $\eps_0$ is small enough.
A well-known classical observation of Freiman~\cite{freiman} now implies that $H \coloneqq S - S$ is a group. We recall the short proof here. 

For any $x, y \in S$, $x - S$ and $y - S$ both lie in $S - S$ and so $|(x - S) \cap (y - S)| > \frac{1}{2}|S|$. That is, there are $> \frac{1}{2}|S|$ pairs $(u,v) \in S \times S$ such that $x - u = y - v$. For each such pair, we have $x - y = u - v$. Now let $x', y' \in S$ be any other elements. Similarly, there are $> \frac{1}{2}|S|$ pairs $(u', v') \in S \times S$ such that $x' - y' = u' - v'$. There are $> \frac{1}{2}|S|$ values of $v$ and $> \frac{1}{2}|S|$ values of $u'$, and all these values lie in $S$; therefore we must have $v = u'$ for some pair of these elements. It then follows that $(x - y) + (x' - y') = (u - v) + (u' - v') = u - v' \in S - S$. Since $x, y, x', y'$ were arbitrary, it follows that $S - S$ is closed under addition. Since it contains $0$ and is closed under taking inverses, it must be a group.

From~\eqref{dist-proj} (noting that $S$ is contained in a single coset of $H$) and~\eqref{prelim-2} we have
\[
  d_{\ent}(U_H, U_S) = \frac{1}{2} \log \frac{|H|}{|S|} \ll \eps^{1/2} \log \frac{1}{\eps}.
\]
Therefore from~\eqref{prelim-1} and~\eqref{triangle} we have
\begin{equation}%
  \label{prelim-3}
  d_\ent(U_H, Y) \ll \eps^{1/2} \log \frac{1}{\eps}.
\end{equation}
This is weaker than Proposition~\ref{sec3-prop} only in the non-linear dependency on $\eps$, which in many applications is not important.%
\vspace*{8pt}

We will deduce the stronger statement of Proposition~\ref{sec3-prop} by bootstrapping this bound. To do this, we require two lemmas which are essentially special cases of Proposition~\ref{sec3-prop} itself. First, we consider the case in which $X$ is highly concentrated near one point. 

\begin{lemma}\label{lem-11}
  There is $\delta_0 > 0$ such that the following is true.
  Suppose that $X$ is a $G$-valued random variable and $x_0 \in G$ is a value such that $\P(X = x_0) \geq 1 - \delta_0$. Then $\H(X) \leq 2 d_{\ent}(X, X)$.
\end{lemma}
\begin{proof}
  By replacing $X$ by $X - x_0$ if necessary, we may assume without loss of generality that $x_0 = 0$.
  Let $X_1,X_2$ be independent copies of $X$. Then our task is equivalent to showing that 
  \begin{equation}\label{eq-form} \H(X_1 - X_2) \geq \frac{3}{2} \H(X).\end{equation}

  Write $p \coloneqq \P(X \ne 0)$ (thus $p \leq \delta_0$), and let $A$ denote the indicator function of the event that $X_1,X_2 \neq 0$; then $\P(A=0) = 1-p^2$ and $\P(A=1) = p^2$.
  As a consequence, we have
  \begin{align} \nonumber
    \H(X_1-& X_2) \geq \H(X_1-X_2|A) \\ \nonumber
               &= (1-p^2) \H(X_1-X_2 | A = 0) + p^2 \H(X_1-X_2 | X_1,X_2 \ne 0) \\
               &\geq (1-p^2) \H(X_1-X_2 | A = 0) + p^2 \H(X | X \ne 0)	\label{star-5}					
  \end{align}
	where we used~\eqref{simple-sumset} in the last line.
	
Now note that for any $z$, if $A=0$ and $X_1-X_2=z$ then $(X_1,X_2)$ can take only two values $(z,0)$ and $(0,-z)$ if $z \neq 0$, and only one value $(0,0)$ if $z=0$.  Hence
   \begin{align*}
	\H(X_1,& X_2 | A=0) - \H(X_1 -X_2 | A=0) \\ & =
	  \H(X_1,X_2 | X_1-X_2, A = 0) \\
		&\leq \P(X_1-X_2 \neq 0|A=0) \log 2 = \frac{2 p(1-p)}{1-p^2} \log 2.
  \end{align*}
	Combining with~\eqref{star-5}, we obtain
		\begin{align}\nonumber
	  \H(X_1-X_2) \geq (1-p^2) \H(& X_1, X_2 | A = 0) \\ & + p^2 \H(X | X \ne 0)- 2p(1-p) \log 2.\label{hx2}
	\end{align}
  We also observe that
  \begin{align}
    2 & \H(X) = \H(X_1,X_2) = \H(X_1,X_2,A)  = \H(X_1, X_2 | A) + \H(A) \nonumber \\
           &= (1-p^2) \H(X_1,X_2 | A=0) + p^2 \H(X_1,X_2 | A = 1) + \h(p^2) \nonumber \\
           &= (1-p^2) \H(X_1,X_2 | A=0) + 2 p^2 \H(X | X \ne 0) + \h(p^2). \label{5point1}
  \end{align}
  We further note that, writing $I$ for the indicator of $X \neq 0$, 
  \begin{equation}\label{5point2} \H(X) = \H(X | I) + \H(I) = p \H(X | X \neq 0) + \h(p).\end{equation}
  Taking~\eqref{5point1} minus $p$ times~\eqref{5point2} gives
 \begin{align*} (2 -p) \H(X) = ( 1-p^2 &) \H( X_1,X_2 | A=0) +  \\ & + p^2 \H(X | X \ne 0) + \h(p^2) - p \h(p).\end{align*}
	Combining this with~\eqref{hx2}, we obtain
	\begin{equation}\label{5point3}
    \H(X_1-X_2) \geq (2-p) \H(X) + p \h(p) - 2 p(1-p) \log 2 - \h(p^2).
  \end{equation}
  Recall that our aim is to demonstrate~\eqref{eq-form}. To get this from~\eqref{5point3}, we first note that expansion to leading order gives
  \begin{equation}\label{5point4}   2 p(1-p) \log 2 + \h(p^2) \leq \left(\frac{1}{2}-p\right) \h(p)\end{equation} provided $\delta_0$ is small enough: the LHS here is $\sim 2p \log 2$, whilst the right hand side is $\sim \frac{1}{2} p \log \frac{1}{p}$. (A more careful analysis shows that $\delta_0 = \frac{1}{20}$ is sufficient.) We also have 
  \begin{equation}\label{5point5} \h(p) = \H(I) \leq \H(X).\end{equation}

The desired bound~\eqref{eq-form} then follows immediately from~\eqref{5point3},~\eqref{5point4} and~\eqref{5point5}.
\end{proof}

\begin{remark*}
  The constant $2$ in the statement of Lemma~\ref{lem-11} can be replaced by anything larger than $1$, at the expense of making $\delta_0$ smaller. This may be shown with very minor modifications of the above argument.
\end{remark*}

We next consider the case of a random variable supported on $H$.
\begin{lemma}%
  \label{lem-12}
  Suppose that $X$ is an $H$-valued random variable with $\H(X) \geq \log |H| - \frac{1}{8}$. Then
  \[
    \log |H| - \H(X) \leq 2 d_\ent(X,X).
  \]
\end{lemma}

To prove this we will use the following lemma concerning couplings of almost uniform random variables, which is plausibly of independent interest.
Here, for a probability distribution $p$ on a group $H$, we write $\Vert p - u_H\Vert_1 := \sum_{x \in H} \bigl\lvert p(x) - \frac{1}{|H|} \bigr\rvert$ for the $\ell^1$-distance of $p$ from the uniform distribution (or, equivalently, twice the total variation distance of $p$ from the uniform distribution).
\begin{lemma}%
  \label{lem:lp}
  Suppose $p_1,p_2,p_3 \colon H \to \R_{\geq 0}$ are three probability distributions on $H$ such that
\begin{equation}\label{total-tv-bound}
\Vert p_1 - u_H \Vert_1 + \Vert p_2 - u_H\Vert_1 + \Vert p_3 - u_H \Vert_1 \leq 1.
\end{equation}
  Then there exists a pair of random variables $(X,Y)$ on $H$ \textup{(}not necessarily independent\textup{)} having the marginal distributions $p_X = p_1$, $p_Y = p_2$ and $p_{X-Y} = p_3$. 
\end{lemma}
\begin{proof}%
  We wish to show that the triple of distributions $(p_1,p_2,p_3) \in \R^{H} \times \R^H \times \R^H$ lies in the convex hull of the set $\Sigma := \{ (\delta_x,\delta_y,\delta_{x-y}) : x,y \in H\} \subseteq \R^H \times \R^H \times \R^H$. Here (as usual) $\delta_t(u) = 1$ if $u = t$, and $\delta_t(u) = 0$ otherwise.  By the (finite-dimensional) Hahn--Banach theorem, this is equivalent to showing that there is no hyperplane separating $(p_1, p_2, p_3)$ from $\Sigma$, or in other words whenever $f_1,f_2,f_3 \colon H \to \R$ are functions such that
\begin{equation}\label{fhyp}
f_1(x)+f_2(y)+f_3(x-y) \geq 0
\end{equation}
for all $x,y \in H$, one also has
\begin{equation}\label{fconc}
 \sum_{x \in H} f_1(x) p_1(x) + \sum_{y \in H} f_2(y) p_2(y)  + \sum_{z \in H} f_3(z) p_3(z) \geq 0.
\end{equation}
Henceforth, assume~\eqref{fhyp}. Note that~\eqref{fhyp} and~\eqref{fconc} are both unaffected if we shift $f_1, f_2, f_3$ by constants $c_1,c_2,c_3$ summing to zero.  Thus we may normalize so that
\[ \min f_1 = \min f_2 = \min f_3.\]
If this quantity is non-negative then~\eqref{fconc} is immediate, so we may assume that it is negative.  By rescaling we may thus normalize so that
\begin{equation}\label{minf}
 \min f_1 = \min f_2 = \min f_3 = -1.
\end{equation}
In particular, there exists $x_0 \in H$ such that $f_1(x_0) = -1$.  From~\eqref{fhyp} and~\eqref{minf}, we conclude that for every $y \in H$ one has
\[ f_2(y) + f_3(x_0-y) \geq 1 \quad \mbox{and} \quad\min(f_2(y), f_3(x_0-y)) \geq -1.\]
This implies that
\begin{align*}
f_2&(y) p_2(y) + f_3(x_0-y) p_3(x_0-y) \\
&\geq (f_2(y)+f_3(x_0-y)) \min(p_2(y),p_3(x_0-y)) \\
&\quad\quad+ \min(f_2(y), f_3(x_0-y)) |p_2(y) - p_3(x_0-y)| \\
&\geq \min(p_2(y),p_3(x_0-y)) - |p_2(y) - p_3(x_0-y)| \\
&= \frac{p_2(y) + p_3(x_0-y)}{2} - \frac{3}{2} \bigl\lvert p_2(y) - p_3(x_0-y) \bigr\rvert \\
&\geq \frac{p_2(y) + p_3(x_0-y)}{2} - \frac{3}{2} \biggl\lvert p_2(y) - \frac{1}{|H|} \biggr\rvert
- \frac{3}{2} \biggl\lvert p_3(x_0-y) - \frac{1}{|H|} \biggr\rvert.
\end{align*}
Summing over $y$, we conclude that
\[  \sum_{y \in H} f_2(y) p_2(y) + \sum_{z \in H} f_3(z) p_3(z) \geq 1 - \frac{3}{2} \Vert p_2 - u_H \Vert_1 - \frac{3}{2} \Vert p_3 - u_H \Vert_1.\]
Cyclically permuting the roles of $f_1,f_2,f_3$ and $p_1,p_2,p_3$ and averaging, the desired bound~\eqref{fconc} then follows from~\eqref{total-tv-bound}.
\end{proof}

\begin{proof}[Proof of Lemma~\ref{lem-12}]%
  By~\eqref{dkus} and Pinsker's inequality~\eqref{pinsker-eq} it follows that
  \begin{equation}\label{tv-bound}
   \Vert p_X - u_H \Vert_1 \leq \sqrt{2(\log |H| - \H(X))} \leq \frac{1}{2}.
  \end{equation}
Applying Lemma~\ref{lem:lp} (with $p_1=p_2=p_X$ and $p_3=\frac{1}{|H|}$), it follows that there exists a pair of random variables $(X_1,X_2)$ such that $X_1,X_2$ each have the same marginal distribution as $X$, and $X_1-X_2$ is uniform on $H$.

  Finally, Lemma~\ref{improv-ruzsa} gives
  \begin{align*}
   \log |H| =  \H(X_1-X_2) & \leq \H(X) + d^*_\ent(X,X) \\ & \leq \H(X) + d_\ent(X,X) + d_\ent(X,X),
  \end{align*}
  which immediately implies the result.
\end{proof}

\begin{proof}[Proof of Proposition~\ref{sec3-prop}] We first establish the case $X = Y$ (with the constant $12$ replaced by $6$). Suppose, as we have throughout the section, that $d_{\ent}(X, X) = \eps \leq \eps_0$. Let $\pi : G \rightarrow G/H$ be the quotient projection.  Recall from~\eqref{dist-proj} that 
  \begin{equation}\label{ento}
    d_\ent(X, U_H) = \H(\pi(X)) + \frac{1}{2} \bigl(\log |H| - \H(X)\bigr).
  \end{equation}
From~\eqref{prelim-3} we have the weak bound $d_\ent(X, U_H) \ll \eps^{1/2} \log \frac{1}{\eps}$.  Thus
  \begin{equation}\label{weak-suff-1} \H(\pi(X)) \ll \eps^{1/2} \log \frac{1}{\eps}\end{equation} and 
	\begin{equation}\label{weak-suff-2} \H(X) \geq \log |H| - O\biggl(\eps^{1/2} \log \frac{1}{\eps}\biggr).\end{equation}

  Now by Proposition~\ref{projections-1} (replacing $H$ there by $G/H$, and recalling that $d_{\ent}(X, X) = \eps$) we obtain
  \begin{equation}\label{pix-d} d_{\ent}\bigl(\pi(X), \pi(X)\bigr) \leq \eps\end{equation} 
and
  \begin{equation}%
    \label{paxy}
    \sum_{y_1, y_2 \in G/H} p_{\pi(X)}(y_1) p_{\pi(X)}(y_2) d_{\ent} (X_{y_1}, X_{y_2}) \leq \eps,
  \end{equation}
	where $X_y$ denotes $X$ conditioned to the event $\pi(X) = y$.  
	
By~\eqref{weak-suff-1} and~\eqref{px-lower}, we see that there is some $y_0 \in G/H$ such that $\P( \pi (X) = y_0) \geq 1 - O\bigl(\eps^{1/2} \log \frac{1}{\eps}\bigr)$. By translating $X$ if necessary, we may assume without loss of generality that $y_0 = 0$, that is to say
\begin{equation}\label{xh} p_{\pi(X)}(0) = \P(X \in H) \geq 1 - O\biggl(\eps^{1/2} \log \frac{1}{\eps}\biggr) \geq \max \biggl(\frac{10}{11}, 1 - \delta_0\biggr) \end{equation}
where $\delta_0$ is the constant from Lemma~\ref{lem-11}, and we assume that $\eps_0$ from the statement of Proposition~\ref{sec3-prop} is sufficiently small.

Applying Lemma~\ref{lem-11} to $\pi(X)$ using~\eqref{pix-d},~\eqref{xh}, we conclude that
\begin{equation}\label{suff-1}
 \H(\pi(X)) \leq 2\eps.
\end{equation}

Meanwhile, discarding all terms in the sum over $y_1$ in~\eqref{paxy} except the term $y_1 = 0$, and using~\eqref{xh}, it follows that 
\[  \sum_{y \in G/H} p_{\pi(X)}(y) d_{\ent}(X_0, X_y) \leq 1.1 \eps.\]
  By~\eqref{ent-lower}, this implies that
  \[ \sum_{y \in G/H}  p_{\pi(X)}(y) \left\lvert \H(X_y) - \H(X_0) \right\rvert \leq 2.2 \eps,\]
  and hence by the triangle inequality
  \[ \big| \H(X | \pi(X)) - \H(X_0) \big| \leq 2.2 \eps.\]
	Using $\H(X) = \H(X | \pi(X)) + \H(\pi(X))$ and~\eqref{suff-1}, we conclude that
	\begin{equation}\label{x-xdash} |\H(X) - \H(X_0) | \leq 4.2 \eps.\end{equation}
  In particular, from~\eqref{weak-suff-2} we deduce 
	\begin{equation}\label{Xprime-lower} \H(X_0) \geq \log |H| - O\biggl(\eps^{1/2} \log \frac{1}{\eps}\biggr).\end{equation}

  Now by discarding all terms in~\eqref{paxy} except the one with $y_1 = y_2 = 0$, and using~\eqref{xh}, we have 
  \[ d_{\ent}(X_0, X_0) \leq 1.21 \eps.\]
It follows from Lemma~\ref{lem-12} that $\H(X_0) \geq \log |H| - 2.42 \eps$, and hence by~\eqref{x-xdash} we obtain
\[ \H(X) \geq \log |H| - 6.62 \eps.\]
Combining this with~\eqref{ento} and~\eqref{suff-1} gives $d_{\ent}(X, U_H) \leq 5.31 \eps \le 6 \eps$, which is the statement of Proposition~\ref{sec3-prop} (with a better constant) in the symmetric case $X = Y$.

Finally, we deduce the general case in which $X$ and $Y$ may be different. Suppose now that $d_{\ent}(X, Y) = \eps \leq \eps'_0$, where $\eps'_0 := \eps_0/2$ with $\eps_0$ the constant above. By the triangle inequality, $d_{\ent}(X, X) \leq 2\eps \leq \eps_0$, and so by the symmetric case of Proposition~\ref{sec3-prop} established above we have $d_{\ent}(X, U_H) \leq 12 \eps$ for some subgroup $H \leq G$. Similarly, we have $d_{\ent}(Y, U_{H'}) \leq 12 \eps$ for some subgroup $H' \leq G$. 

It remains to argue that $H = H'$. For this, we observe that by the triangle inequality we have \begin{equation}\label{uhh} d(U_H, U_{H'}) \leq 25 \eps.\end{equation} If $H \neq H'$, then $H + H'$ is a subgroup of $G$ properly containing $H, H'$ and therefore of size at least $2\max(|H|, |H'|)$. Since $U_H - U_{H'}$ is uniform on $H + H'$, we have $d(U_H, U_{H'}) \geq \log 2$, which contradicts~\eqref{uhh} if $\eps_0$ is small enough. Therefore we do indeed have $H = H'$, and this concludes the proof.
\end{proof}

\section{Skew dimension and a result of P\'alv\"olgyi and Zhelezov}\label{sec5}

In this section we give the proof of Theorem~\ref{pz-actual} (and thus Theorem~\ref{pz-main}).

\begin{proof}[Proof of Theorem~\ref{pz-actual}]
Let $\eps > 0$ be a small constant to be specified later, and set $C \coloneqq  2/\eps$. We will prove Theorem~\ref{pz-actual} with this particular value of $C$.

We proceed by induction on $|A||B|$ and on $D$. Denote by $\pi\colon \Z^D \rightarrow \Z$ projection onto the first coordinate. We may assume that at least one of the sets $\pi(A), \pi(B)$ is not a singleton (otherwise $D$ may be reduced to $D - 1$).

Let $X_1, X_2$ be uniform random variables on $A, B$ respectively, and let $Y_i = \pi(X_i)$.
Applying Proposition~\ref{projections-1} and rearranging, we obtain
\begin{equation}\label{res-prop} \sum_{i, j} p_{Y_1}(i) p_{Y_2}(j) \log \frac{K}{K_{i,j}} \geq d_\ent(Y_1,Y_2)\end{equation}
where
\[
  K \coloneqq \exp\bigl( d_\ent(X_1,X_2) \bigr)
\]
and
\[
  K_{i,j} \coloneqq \exp\bigl( d_\ent((X_1|Y_1=i), (X_2|Y_2=j)) \bigr).
\]
We now divide into two cases, according to whether $d_\ent(Y_1, Y_2) \leq \eps$ or not.
\vspace*{8pt}

\noindent
\emph{Case 1:} $d_\ent(Y_1, Y_2) \leq \eps$.
Let $\eps_0$ be the constant from Proposition~\ref{sec3-prop}, and assume that $\eps \leq \min\bigl(\eps_0,\frac{1}{24}\bigr)$.
By Proposition~\ref{sec3-prop} and the fact that $H = \{0\}$ is the only finite subgroup of $\Z$, we have
\[
  d_{\ent}(Y_1, 0), d_{\ent}(Y_2, 0) \leq 12 d_\ent(Y_1,Y_2).
\]
Since $d_{\ent}(Y_i, 0) = \frac{1}{2} \H(Y_i)$, it follows that
\begin{equation}\label{hy1y2}  
  \H(Y_1) + \H(Y_2) \leq 48 d_\ent(Y_1,Y_2) \leq C d_\ent(Y_1,Y_2)
	\end{equation}
by the choice of $C, \eps$.
Inserting this into~\eqref{res-prop} and rearranging, we obtain
\[
  \sum_{i, j} p_{Y_1}(i) p_{Y_2}(j) \log \left( \frac{K_{i,j}}{K (p_{Y_1}(i) p_{Y_2}(j))^{1/C}}  \right)\leq 0.
\]
In particular, there exist $i,j$ such that
\[
K_{i,j} \leq K (p_{Y_1}(i) p_{Y_2}(j))^{1/C} \leq K.
\]
Invoking the induction hypothesis (with $A, B$ replaced by $A \cap \pi^{-1}(\{i\})$ and $B \cap \pi^{-1}(\{j\})$ respectively), we see that there are sets $A' \subseteq A, B' \subseteq B$ with 
\[
  \dim_* A', \dim_* B' \leq C \log K_{i,j} \leq C \log K
\]
and 
\begin{align*}
  |A'||B'| &\geq K_{i,j}^{-C} |A \cap \pi^{-1}(\{i\})||B \cap \pi^{-1}(\{j\})| 
   = K_{i,j}^{-C} p_{Y_1} (i) p_{Y_2}(j) |A||B| \\
  &\geq K_{i,j}^{-C}  \biggl(\frac{K_{i, j}}{K}\biggr)^{C} |A| |B| = K^{-C} |A||B|.
\end{align*} 
This closes the induction in Case 1.\vspace*{8pt}

\noindent
\emph{Case 2:} $d_\ent(Y_1, Y_2) \geq \eps$. 

In this case we see from~\eqref{res-prop} that
\[ \sum_{i,j} p_{Y_1}(i) p_{Y_2}(j) \log \frac{K}{ K_{i,j}} \geq \eps.\]
 The contribution from those $(i,j)$ with $\log \frac{K}{K_{i,j}} \leq \frac{\eps}{2}$ is at most $\frac{\eps}{2}$. Thus if we set
 \[ S \coloneqq  \biggl\{ (i,j) : \log \frac{K}{K_{i,j}} > \frac{\eps}{2} \biggr\}\] 
then
\begin{equation}\label{ond}
 \sum_{(i,j) \in S}p_{Y_1}(i) p_{Y_2}(j) \log \frac{K}{ K_{i,j}} \geq \frac{\eps}{2}.
\end{equation}
Note in particular that, if $(i,j) \in S$, 
\begin{equation}\label{kij} K_{i,j} < K.\end{equation}
By the induction hypothesis, for each pair $(i,j) \in S$ there are sets $A'_{i,j} \subseteq A$, $B'_{i,j} \subseteq B$ with
\begin{equation}\label{ind} |A'_{i,j}||B'_{i,j}| \geq K_{i,j}^{-C} p_{Y_1}(i) p_{Y_2}(j) |A| |B|,\end{equation}
and all with skew-dimension at most
 \begin{equation}\label{qd-bd} C\log K_{i,j}  \leq C\biggl(\log K  - \frac{\eps}{2}\biggr) = C \log K - 1.\end{equation} (Here we used the fact that $C = 2/\eps$.)
 
 For each $i \in \Z$, set $A'_{i}$ to be the largest of the sets $A'_{i, j}$, $(i, j) \in S$ (or $A'_i=\emptyset$ if $(i,j) \notin S$ for every $j$), and similarly for each $j \in \Z$ set $B'_j$ to be the largest of the sets $B'_{i, j}$, $(i, j) \in S$, breaking ties arbitrarily.
 Finally, set $A' \coloneqq  \bigcup_{i \in \Z} A'_i$ and $B' \coloneqq  \bigcup_{j \in \Z} B'_j$. By the definition of skew-dimension and the bound~\eqref{qd-bd}, we have $\dim_* A', \dim_* B' \leq C \log K$. 
 
From the elementary inequality $t \geq \log t$ for $t \geq 1$ applied to $t = (K/K_{i,j})^C$ (noting by~\eqref{kij} that we do indeed have $t \geq 1$), we have
\[
  K_{i,j}^{-C} \geq C K^{-C} \log \frac{K}{K_{i,j}}
\]
for any $(i,j) \in S$.  From this and~\eqref{ind},~\eqref{ond} we have
 \begin{align*}
 |A'| |B'|  & = \sum_{(i, j) \in \Z^2} |A'_i| |B'_j| \\
&\geq \sum_{(i,j) \in S} |A'_{i,j}| |B'_{i,j}| \\ 
&\geq |A| |B|\sum_{(i,j) \in S} K_{i,j}^{-C} p_{Y_1}(i) p_{Y_2}(j)  \\
&\geq |A| |B|\sum_{(i,j) \in S} \biggl( C K^{-C}  \log \frac{K}{K_{i,j}} \biggr)p_{Y_1}(i) p_{Y_2}(j)  \\
& \geq \frac{C\eps}{2} K^{-C} |A| |B| = K^{-C} |A||B|.
  \end{align*}
This completes the induction, and the theorem is proved.
\end{proof}

\section{Dimension and a result of the second author}\label{sec-manners}

We turn now to the question of the dimension (as opposed to the weaker skew-dimension) of subsets of $\Z^D$ with small doubling. Our aim in this section is to give an entropic proof of Theorem~\ref{manners-thm}. In so doing, we will also lay the groundwork for the proof of Theorem~\ref{manners-thm-new}, our improvement upon this result.

As in~\cite[Slogan 2.5]{manners}, a key idea is that a set $A \subseteq \Z^D$ with small doubling must look rather singular under the projection map $\phi : \Z^D \rightarrow \F_2^D$. In Lemma~\ref{lem:uneven-mod-2} below, we give an entropic formulation of this principle. We isolate the following lemma from the proof.

\begin{lemma}\label{lem73}
Let $G$ be torsion-free, and let $X, Y$ be $G$-valued random variables. Then $d_{\ent}(X, 2Y) \leq 5 d_{\ent}(X, Y)$.
\end{lemma}

\begin{proof}
  We assume $X,Y$ are independent.  Then\footnote{The use of $d_\ent^\ast$ here simplifies an earlier version of the argument, and was suggested to the authors by Noah Kravitz.}
  \begin{align}%
    \nonumber
    \H(X-2Y) &= \H((X-Y)-Y)  \\
             &\le d^\ast_\ent(X-Y,Y) + \frac12 \H(X-Y) + \frac12 \H(Y)
             \label{eq:dast}
  \end{align}
  by definition of $d_\ent^\ast$.
  By Lemma~\ref{improv-ruzsa},
  \begin{align}%
    \nonumber
    d_\ent^\ast(X-Y,Y) &\le d_\ent(Y,Y) + d_\ent(X-Y,Y) \\
    \label{eq:strong-triangle}
                       &\le 2 d_\ent(X,Y) + d_\ent(X-Y,Y).
  \end{align}
  Letting $Y_1,Y_2$ be independent copies of $Y$ (which are also independent of $X$) we have
  \begin{equation}%
    \label{eq:dxxmy}
    d_\ent(X-Y,Y) = \H(X-Y_1-Y_2) - \frac12 \H(X-Y) - \frac12 \H(Y).
  \end{equation}
  Writing $A \coloneqq Y_1$, $B \coloneqq Y_2$ and $C \coloneqq X- Y_1 - Y_2$, we have
  \[ \H(A,B,C) = \H(X, Y_1, Y_2) = \H(X) + 2 \H(Y),\] and
  \[ \H(A, C) = \H(A, C + A) = \H(Y_1, X - Y_2) = \H(Y) + \H(X - Y_2),\]
  \[ \H(B, C) =\H(B, C + B) =  \H(Y_2, X - Y_1) = \H(Y) + \H(X - Y_1)\]
  so applying the submodularity inequality~\eqref{submodularity} gives
  \[
    \H(X-Y_1-Y_2) \leq \H(X-Y_1) + \H(X-Y_2) - \H(X).
  \]
  Combining this with~\eqref{eq:dxxmy} gives
  \[
    d_\ent(X-Y,Y) \le \frac32 \H(X-Y) - \H(X) - \frac12 \H(Y)
  \]
  which, together with~\eqref{eq:dast} and~\eqref{eq:strong-triangle}, yields
  \[
    \H(X-2Y) \le 2 d_\ent(X,Y) + 2 \H(X-Y) - \H(X) = 4 d_\ent(X,Y) + \H(Y)
  \]
  and so
  \[
    d_\ent(X,2Y) \le 4 d_\ent(X,Y) + \frac12 (\H(Y)-\H(X)) \le 5 d_\ent(X,Y)
  \]
  where we used~\eqref{ent-lower} in the last step.

\end{proof}

\begin{lemma}%
  \label{lem:uneven-mod-2}
Let $X, Y$ be $\Z^D$-valued random variables for some $D \geq 0$. Denote by $\phi : \Z^D \rightarrow \F_2^D$ the natural homomorphism. Then 
 \[
    \H(\phi(X)), \H(\phi(Y)) \leq 10 d_\ent(X,Y).
  \]
\end{lemma}

\begin{proof}
By Proposition~\ref{projections-1} and Lemma~\ref{lem73},
\begin{equation}\label{eq983} d_{\ent}(\phi(X), \phi(2Y)) \leq d_{\ent}(X, 2Y) \leq 5 d_{\ent}(X, Y).\end{equation}
However, $\phi(2Y)$ is identically zero and so
\[ d_{\ent}(\phi(X), \phi(2Y)) = d_{\ent}(\phi(X), 0) = \frac{1}{2}\H(\phi(X)).\]
Combining this with~\eqref{eq983} gives the stated bound for $\H(\phi(X))$. The bound for $\H(\phi(Y))$ follows in the same way.
\end{proof}

\begin{remark*}
It is perhaps worth remarking on the meaning and proof of this statement. Supposing that $A \subset \Z^D$ is a set with small (combinatorial) doubling $K$, it follows that the dilate $2 \cdot A$, which is contained in $A + A$, is commensurate (up to polynomial factors in $K$) with $A$. Projecting mod $2$, one therefore expects the projection $\pi(A)$ to be commensurate with the projection $\pi(2 \cdot A) = \{0\}$.
A version of this argument appears in~\cite[Appendix~B]{manners}.
In the entropy setting, Lemma~\ref{lem73} acts as a replacement for the trivial observation that $2 \cdot A$ is contained in $A+A$.
\end{remark*}

Now we are ready for the proof of Theorem~\ref{manners-thm} itself. To make the argument work, we will in fact need to establish the following bipartite variant of the result.

\begin{theorem}%
  \label{bil-manners-thm}
  There is an absolute constant $\CC$ such that, setting $f(t) \coloneqq  \CC t(1 + t)$, the following is true. Let $D \in \N$, and suppose that $A, B \subseteq \Z^D$ are finite non-empty sets. Then there exist nonempty $A' \subseteq A$, $B' \subseteq B$ with 
  \begin{equation}%
    \label{a4-new}
    \log \frac{|A|}{|A'|} + \log \frac{|B|}{|B'|} \leq f\bigl(d_{\ent}(U_A, U_B)\bigr)
  \end{equation}
  and such that 
  $\dim A', \dim B' \leq \CC d_{\ent}(U_A, U_B)$.
\end{theorem}
It is clear from~\eqref{structure-ineqs} that this indeed implies Theorem~\ref{manners-thm}.

We first prove a simple lemma which will be used several times in what follows.
\begin{lemma}%
  \label{lem:hom-lemma}
  Let $\phi \colon G \to H$ be a homomorphism, and $A,B \subseteq G$ finite subsets.
  For $x,y \in H$ write $A_x = A \cap \phi^{-1}(x)$ and $B_y = B \cap \phi^{-1}(y)$ for the fibres of $A$ and $B$, and write $\alpha_x \coloneqq \frac{|A_x|}{|A|}$ and $\beta_y \coloneqq \frac{|B_y|}{|B|}$.
  Write $k = d_\ent(U_A,U_B)$, $\overline{k} = d_\ent(\phi(U_A),\phi(U_B))$ and $M=\H(\phi(U_A))+\H(\phi(U_B))$. Then there exist $x,y \in H$ such that $A_x, B_y$ are non-empty and with
  \begin{equation}%
    \label{eq:hom-sampling}
   \overline{k} \log \frac{1}{\alpha_x \beta_y} \leq M\big(k - d_\ent(U_{A_x},U_{B_y}) \big).
  \end{equation}
\end{lemma}
\begin{proof}
  First observe that the random variables $(U_A | \phi(U_A) = x)$ and $(U_B | \phi(U_B) = y)$ are equal in distribution to $U_{A_x}, U_{B_y}$ respectively, that is to say the uniform distributions on the fibres. 
  It follows from Proposition~\ref{projections-1} that
  \begin{equation}%
    \label{projy-1}
    \sum_{x, y \in H} \alpha_x \beta_y d_\ent(U_{A_x}, U_{B_y}) \leq k - \overline{k}.
  \end{equation}
  By definition, $M = \sum_{x,y} \alpha_x \beta_y \log \frac1{\alpha_x \beta_y}$  and hence 
  \[
    \sum_{x, y \in H} \alpha_x \beta_y \big( M d_\ent(U_{A_x}, U_{B_y}) + \overline{k} \log \frac1{\alpha_x \beta_y} \big) \leq Mk.
  \]
  It follows by the pigeonhole principle that there is at least one choice of $x,y$ such that $\alpha_x,\beta_y>0$ and
  \[
    M d_\ent(U_{A_x}, U_{B_y}) + \overline{k} \log \frac1{\alpha_x \beta_y} \leq Mk.
  \]
  Rearranging gives~\eqref{eq:hom-sampling}.
\end{proof}

We now turn to the proof of Theorem~\ref{bil-manners-thm}.
\begin{proof}[Proof of Theorem~\ref{bil-manners-thm}]
  Let us begin by noting the simple inequality
  \begin{align}%
    f(b) &= \CC b (1+b) \nonumber \\
         &\leq \CC b (1+a) = f(a) - \CC (a-b) (1+a) \label{f-ineqs}
  \end{align}
  for all $a,b \in \R$ with $0 \leq b \leq a$.

  Let us turn now to the main proof. We will proceed by induction on $|A| + |B|$.
  We may also assume that $A, B$ do not sit inside cosets of a proper subgroup of $\Z^D$, else we may replace $\Z^D$ by that subgroup.
  We also suppose $D \geq 1$, as the result is trivial otherwise.

  Let $\phi : \Z^D \rightarrow \F_2^D$ be the natural homomorphism.
  Then, by the preceding remark and the fact that $\ker \phi$ is a proper subgroup of $\Z^D$, we may assume that at least one of $\phi(A), \phi(B)$ is not a singleton.
  For $x, y \in \F_2^{D}$, denote by $A_x \coloneqq  A \cap \phi^{-1}(x)$ and $B_y \coloneqq  B \cap \phi^{-1}(y)$ the fibres of $A, B$.
  Note that
  \begin{equation}\label{fibre-decrement-a} |A_x| + |B_y| < |A| + |B|\end{equation}
  for all $x, y$.

  Write $k \coloneqq d_{\ent}(U_A, U_B)$ and
  $\eps \coloneqq d_{\ent}(\phi(U_A), \phi(U_B))$.
  Let $\delta>0$ be a small positive constant to be determined later, and set $C_1 := \max(20/\delta, 100)$.
  We will divide into two cases, according to whether or not $\eps \leq \delta$.%
  \vspace*{8pt}

  \noindent
  \emph{Case 1:} $\eps > \delta$.
  By Lemma~\ref{lem:uneven-mod-2} with $X = U_A$, $Y = U_B$, we have
  \begin{equation}\label{projy-2}
    \H(\phi(U_A)) + \H(\phi(U_B)) \leq 20k.
  \end{equation}
  By Lemma~\ref{lem:hom-lemma} applied to $G = \ZZ^D$ and $H = \F_2^D$, we may find $x,y \in \F_2^D$ such that~\eqref{eq:hom-sampling} holds. Fix such $x, y$, and for brevity set $k' \coloneqq  d_{\ent}(U_{A_x}, U_{B_y})$.  Then~\eqref{eq:hom-sampling} implies that $k' \leq k$ and
  \begin{equation}%
    \label{eq:m-ineq}
    \log \frac{|A|}{|A_x|} + \log \frac{|B|}{|B_y|} \leq \frac{20 k}{\eps} (k - k').
  \end{equation}
  Noting~\eqref{fibre-decrement-a}, we may apply the induction hypothesis to conclude that there are $A' \subseteq A_x$, $B' \subseteq B_y$ with
  \[
    \log \frac{|A_x|}{|A'|} + \log \frac{|B_x|}{|B'|} \leq f(k')
  \]
  such that $\dim A', \dim B' \leq \CC k' \leq \CC k$. This and~\eqref{eq:m-ineq} immediately imply
  \[  \log \frac{|A|}{|A'|} + 
    \log \frac{|B|}{|B'|} \leq f(k') + \frac{20k}{\eps}(k - k').\]
    By~\eqref{f-ineqs}, this is at most $f(k)$, since $C_1 \geq 20/\delta \geq 20/\eps$.
  This closes the induction in Case 1.%
  \vspace*{8pt}

  \noindent
  \emph{Case 2:} $\eps \leq \delta$. Recall here that $\eps = d_\ent(\phi(U_A), \phi(U_B))$, and note that $\eps \leq k$ by Proposition~\ref{projections-1}.
  Let $\eps_0$ be the constant from Proposition~\ref{sec3-prop} and suppose $\delta \leq \eps_0$.
  By Proposition~\ref{sec3-prop} there is some $H \leq \F_2^D$ such that 
  \[
    d_{\ent}(\phi(U_A), U_H), d_{\ent}(\phi(U_B), U_H) \leq 12 \eps.
  \]
  It is possible that $H = \F_2^D$. In this case, we have by~\eqref{ent-lower} and Lemma~\ref{lem:uneven-mod-2} that 
  \begin{align*}
    \log (2^D) =  \H(U_H) \leq \H(\phi(U_A)) + 2 d_{\ent}(\phi(U_A), U_H) \leq 10k + 24 \eps \leq 34k,
  \end{align*}
  and so $D \leq 100k$. This gives Theorem~\ref{bil-manners-thm} simply by taking $A= A'$, $B = B'$, since $\CC \geq 100$.

  Alternatively, suppose that $H$ is a proper subgroup of $\F_2^D$. Denote by $\tilde \phi$ the composition of $\phi$ with projection to $\F_2^D/H$. By~\eqref{ent-proj-dist} we have
  \[
    \H(\tilde\phi(U_A)) \leq 2 d_{\ent} (\phi(U_A), U_H) \leq 32 \eps.
  \]
  By~\eqref{px-lower} there is some $x_0$ such that $\P(\tilde \phi(U_A) = x_0) \geq e^{-32 \eps} \geq e^{-32 \delta}$.
  Choosing $\delta$ sufficiently small, this is $\geq 1 - \delta_0$ where $\delta_0$ is the constant in Lemma~\ref{lem-11}, and so by Lemma~\ref{lem-11}
  \[
    \H(\tilde\phi(U_A)) \leq 2 d_\ent \bigl(\tilde \phi(U_A), \tilde \phi(U_A)\bigr) \leq 4 d_\ent(\tilde \phi(U_A), \tilde \phi(U_B))
  \]
  where the second inequality is by~\eqref{triangle}.
  The same bound holds for $\H(\tilde\phi(U_B))$.

  Hence by Lemma~\ref{lem:hom-lemma} applied to $\tilde \phi$, $A$ and $B$ (noting that we cannot have $\H(\tilde\phi(U_A)) = \H(\tilde\phi(U_B)) = 0$, as then $A,B$ would be contained in cosets of a proper subgroup) we deduce that there exist $x \in \F_2^D / H$, $y \in \F_2^D / H$ such that
  \begin{equation}%
    \label{new-hom-bound}
    \log \frac{|A|}{|A_x|} + \log \frac{|B|}{|B_y|} \leq 8 \bigl(k - d_\ent(U_{A_x}, U_{B_y})\bigr),
  \end{equation}
  where $A_x = A \cap \tilde\phi^{-1}(x)$, $B_y = B \cap \tilde\phi^{-1}(y)$.
  
  We now finish the proof as before.
  Set $k' = d_\ent(U_{A_x}, U_{A_y})$, which is $\leq k$ by~\eqref{new-hom-bound}.
  Since $A,B$ are not contained in cosets of a proper subgroup of $\ZZ^D$, we have
  \[
    |A_x| + |B_y| < |A| + |B|
  \]
  and so by induction we may find $A' \subseteq A_x$, $B' \subseteq B_y$ with
  \[
    \log \frac{|A_x|}{|A'|} + \log \frac{|B_y|}{|B'|} \leq f(k')
  \]
  and $\dim A',\dim B' \leq \CC k' \leq \CC k$.
  Combining with~\eqref{new-hom-bound} gives
  \[ \log \frac{|A|}{|A'|} + \log \frac{|B|}{|B'|} \leq f(k') + 8(k - k').\]
  By~\eqref{f-ineqs} (and since $C_1 \geq 8$) this is at most $f(k)$.  This closes the induction in Case 2 and the proof of Theorem~\ref{bil-manners-thm}  is complete.
\end{proof}

\begin{remark*} For this argument, the full strength of Proposition~\ref{sec3-prop} was not needed, and the weaker bound~\eqref{prelim-3} would have sufficed.
\end{remark*}

\section{Entropy formulation of PFR over \texorpdfstring{$\F_2$}{F2}}

In this section we establish Proposition~\ref{pfr-f2-prop}. Recall that the content of this proposition is that the following two statements are equivalent:

\noindent
\emph{Statement 1.} If $A \subseteq \F_2^D$ and if $\sigma[A] \leq K$ then $A$ is covered by $O( K^{O(1)} )$ cosets of some subspace $H \leq \F_2^D$ of size at most $|A|$.

\noindent
\emph{Statement 2.} If $X, Y$ are two $\F_2^D$-valued random variables, there is some subgroup $H \leq \F_2^D$ such that $d_{\ent}(X, U_H), d_{\ent}(Y, U_H) \ll d_{\ent}(X, Y)$.

\begin{proof}[Proof of Proposition~\ref{pfr-f2-prop}]
We first derive the entropic statement, that is to say Statement 2 above, from the combinatorial one (Statement 1).
Write $k \coloneqq d_{\ent}(X, Y)$ and set $K \coloneqq e^k$. We may assume that $k \geq \eps_0$, where $\eps_0$ is the constant in Proposition~\ref{sec3-prop}, since the claim follows immediately from that proposition otherwise.
Applying Proposition~\ref{sec2-prop} with $C = 4$, we obtain a set $S \subseteq \F_2^D$ with 
\begin{equation}\label{ent-s-y} d_\ent(X, U_S) \ll k\end{equation}
and (recalling that $\frac{|S+S|}{|S|} \leq \left( \frac{|S-S|}{|S|} \right)^3$; see e.g.~\cite[Corollary~2.12]{tao-vu})
\begin{equation}\label{s-doubling} |S + S| \ll K^{O(1)}|S|.\end{equation}
By Statement 1 there is a subgroup $H \leq \F_2^D$, $|H| \leq |S|$, such that $S$ is covered by $O(K^{O(1)})$ cosets of $H$. Note, in particular, that $S + H$ is contained in the union of the aforementioned cosets, and so $|S + H| \ll K^{O(1)} \min( |S|, |H|)$. Now for any sets $A, B$ we have
\begin{align*}
  d_{\ent}(U_A, U_B) & = \H(U_A - U_B) - \frac{1}{2}(\H(U_A) + \H(U_B)) \\
                     & \leq \log |A-B| - \frac{1}{2} \bigl(\log |A| + \log |B|\bigr) \\
                     & = \log \left(\frac{|A - B|}{|A|^{1/2} |B|^{1/2}}\right).
\end{align*}
(This is the bipartite version of~\eqref{structure-ineqs}.)
Applying this with $A = S$ and $B = H$ (and noting $H = -H$) gives $d_{\ent}(U_S, U_H) \ll k$, and so by the triangle inequality and~\eqref{ent-s-y} we have $d_{\ent}(X, U_H) \ll k$, which is the conclusion in Statement 2. 

We turn now to the reverse implication, deriving the combinatorial Statement 1 from the entropic Statement 2.
Suppose that $A \subseteq \F_2^D$ is a set and write $K \coloneqq \sigma[A]$ and $k \coloneqq \log K$.
Then, by~\eqref{structure-ineqs}, we have $d_{\ent}(A, -A) = \log \sigma_{\ent}[A] \leq k$.
Assuming Statement 2, there is some finite subgroup $H \leq \F_2^D$ with $d_{\ent}(U_A, U_H) \ll k$.
By~\eqref{ent-lower} and the fact that $\H(U_A) = \log |A|$, $\H(U_H) = \log |H|$, we have 
\begin{equation}\label{h-size-bds} K^{-O(1)} |A| \ll |H| \ll K^{O(1)} |A|.\end{equation}
Writing $p(x)$ for the density function of $U_A - U_H$, thus $p(x) = \frac{|A \cap (H + x)|}{|A||H|}$, it follows from~\eqref{px-lower} that there is some $x_0$ such that 
\[
  p(x_0) \geq e^{-\H(U_A - U_H)} = e^{-d_{\ent}(U_A, U_H)} |A|^{-1/2} |H|^{-1/2}  \gg K^{-O(1)} |A|^{-1},
\]
or in other words $|A \cap (H + x_0)| \gg K^{-O(1)} |H|$. 

Recall the Ruzsa covering lemma (see e.g.,~\cite[Lemma 2.14]{tao-vu}), which states that if $|U + V| \leq K|U|$ then $V$ is covered by $K$ translates of $U - U$. Applying this with $U = A \cap (H + x_0)$ and $V = A$, and using the fact that $U + V \subseteq A + A$ and $U - U \subseteq H$, it follows that $A$ is covered by $O(K^{O(1)})$ translates of $H$. 

If $|H| \leq |A|$, we are done. If $|H| > |A|$, pass to a subgroup $H' \leq H$ of size in the range $(\frac{1}{2}|A|, |A|]$; then $A$ is covered by $O(K^{O(1)})$ translates of $H'$, and the proof is complete in this case also. 
\end{proof}

A minor modification of the first part of the above proof, using the quantity $C_{\PFR}$ from the introduction in place of Statement 1, gives the following statement.

\begin{proposition}\label{combined}
  Let $X, Y$ be $\F_2^D$-valued random variables, and suppose that $d_{\ent}(X, Y) = k$. Then there is some subgroup $H \leq \F_2^D$ such that $d_{\ent}(X, U_H) \leq C k (1 + k^{C_{\PFR} - 1})$, for some absolute constant $C$ \textup{(}which may depend on $C_{\PFR}$\textup{)}. 
\end{proposition}

\section{Dimension and the weak PFR conjecture}\label{sec4}

We now prove Theorem~\ref{manners-thm-new} (and hence Corollary~\ref{pfr-weak-pfr}). The proof is along somewhat similar lines to the proof of Theorem~\ref{manners-thm} given in Section~\ref{sec-manners}, but more involved. An important ingredient will be the following lemma.

Throughout this section, $C$ will be the constant in Proposition~\ref{combined} (but the precise nature of this constant is not important).

\begin{lemma}\label{iterative-small}
Suppose that $X$ and $Y$ are $\F_2^D$-valued random variables. Then there is a subgroup $H \leq \F_2^D$ such that, denoting by $\psi \colon \F_2^D \rightarrow \F_2^D/H$ the natural projection, and setting $k \coloneqq d_{\ent}(\psi(X), \psi(Y))$, we have
\begin{equation}\label{logxyh} \log |H| \leq 2(\H(X) + \H(Y))\end{equation}
and
\begin{equation}\label{entsum-up} \H(\psi(X)) + \H(\psi(Y)) \leq 8C k (1 + k^{C_{\PFR} - 1}).\end{equation}
\end{lemma}

We isolate the following (sub-) lemma from the proof.

\begin{lemma}%
  \label{lem54} Let $n\in \N$. Let $X, Y$ be $\F_2^n$-valued random variables.
  Set $k \coloneqq d_{\ent}(X, Y)$, and suppose that
  \begin{equation}%
    \label{eq:lem54-hypo}
    \H(X) + \H(Y) > 8Ck (1 + k^{C_{\PFR} - 1}).
  \end{equation}
  Then there is a nontrivial subgroup $H \leq \F_2^n$ such that
  \begin{equation}%
    \label{lem54-i}
    \log |H| \leq \H(X) + \H(Y)
  \end{equation}
  and \textup{(}writing $\psi \colon \F_2^n \to \F_2^n / H$ as above\textup{)} 
  \begin{equation}%
    \label{lem54-ii}
    \H(\psi(X)) + \H(\psi(Y)) \leq \frac{1}{2} \bigl(\H(X) + \H(Y)\bigr).
  \end{equation}
\end{lemma}
\begin{proof}
Set $k \coloneqq d_{\ent}(X, Y)$. Applying Proposition~\ref{combined}, we obtain a subgroup $H$ such that $d_\ent(X, U_H), d_\ent(Y, U_H) \leq Ck (1 + k^{C_{\PFR} - 1})$. By~\eqref{ent-proj-dist} and~\eqref{eq:lem54-hypo}, it follows that
\[ \H(\psi(X)) + \H(\psi(Y)) \leq 4C k (1 + k^{C_{\PFR} - 1}) < \frac{1}{2}(\H(X) + \H(Y)),\] which is~\eqref{lem54-ii}. To prove~\eqref{lem54-i}, an application of~\eqref{ent-lower} yields
\[ \log |H| - \H(X) \leq 2 d_{\ent}(X, U_H) \leq 2 Ck (1 + k^{C_{\PFR} - 1}),\] and similarly for $Y$. Therefore using~\eqref{eq:lem54-hypo} we have
\[ \log |H| \leq \frac{1}{2}(\H(X) + \H(Y)) + 2 C k (1 + k^{C_{\PFR} - 1}) <  \H(X) + \H(Y),\] which gives the required bound~\eqref{lem54-i}.

If $H$ were trivial we would have $\psi(X) = X$, $\psi(Y) = Y$ and so~\eqref{lem54-ii} would imply $\H(X) + \H(Y) = 0$, which then contradicts~\eqref{eq:lem54-hypo}.
\end{proof}

\begin{proof}[Proof of Lemma~\ref{iterative-small}]
We iteratively define a sequence $\{0\} = H_0 < H_1 < \cdots $ of subgroups of $\F_2^D$. Denote by $\psi_i : \F_2^D \rightarrow \F_2^D/H_i$ the $i$th associated projection operator, and set $k_i \coloneqq d_{\ent}(\psi_i(X), \psi_i(Y))$. We stop the iteration at the $i$th stage if we have 
\begin{equation}\label{stop-criterion} \H(\psi_i(X)) + \H(\psi_i(Y)) \leq 8C k_i (1 + k_i^{C_{\PFR} - 1}).\end{equation} Otherwise, we apply Lemma~\ref{lem54} to $\psi_i(X), \psi_i(Y)$, obtaining a nontrivial subgroup $H_{i+1}/H_i \leq \F_2^D/H_i$ such that 
\begin{equation}\label{eq47} \log \frac{|H_{i+1}|}{|H_i|} \leq \H(\psi_i(X)) + \H(\psi_i(Y))\end{equation} and
\begin{equation}\label{eq48} \H(\psi_{i+1}(X)) + \H(\psi_{i+1}(Y)) \leq \frac{1}{2}\bigl(\H(\psi_i(X)) + \H(\psi_i(Y))\bigr).\end{equation}
Clearly from iterated application of~\eqref{eq48} we obtain \[ \H(\psi_i(X)) + \H(\psi_i(Y)) \leq 2^{-i} (\H(X) + \H(Y)).\] Then, from a telescoping application of~\eqref{eq47} we get 
\begin{equation}\label{h-bd} \log |H_i| \leq 2 (\H(X) + \H(Y)).\end{equation}
Since the groups $H_i$ form a strictly increasing sequence, the iteration does terminate at some time $i$. At this time we have both~\eqref{stop-criterion} and~\eqref{h-bd} and so, setting $\psi = \psi_i$, the proof of Lemma~\ref{iterative-small} is concluded.\end{proof}

Now we turn our attention to Theorem~\ref{manners-thm-new}. It is a consequence of the following bipartite statement, which should be compared to Theorem~\ref{bil-manners-thm}.

\begin{theorem}\label{thm83}
  There are absolute constants $C_1, C_2$ such that, setting $f(t) \coloneqq  C_1 t (1 + t^{1 - 1/C_{\PFR}})$, the following is true.
  Let $D \in \N$, and suppose $A, B \subseteq \Z^D$ are finite non-empty sets, and set $k \coloneqq d_{\ent}(U_A, U_B)$.
  Then there exist nonempty $A' \subseteq A$, $B' \subseteq B$ with 
  \[
    \log \frac{|A|}{|A'|} + \log \frac{|B|}{|B'|} \leq f(k)
  \]
  and such that $\dim A', \dim B' \leq C_2 k$.
\end{theorem}
\begin{proof}
  We will proceed by induction on $|A| + |B|$.
  We may also assume that $A, B$ do not sit inside cosets of a proper subgroup of $\Z^D$, else we may replace $\Z^D$ by that subgroup.

  Let $\phi : \Z^D \rightarrow \F_2^D$ be the natural homomorphism.
  By Lemma~\ref{lem:uneven-mod-2} we have
  \begin{equation}%
    \label{ent-here}
    \H(\phi(U_A)), \H(\phi(U_B)) \leq 10 k.
  \end{equation}

  Applying Lemma~\ref{iterative-small} to $\phi(U_A), \phi(U_B)$, we find a subgroup $H \leq \F_2^D$ and associated projection $\psi : \F_2^D \rightarrow \F_2^D /H$ such that, denoting by $\tilde\phi = \psi \circ \phi : \Z^D \rightarrow \F_2^D / H$ the natural (composite) projection, we have 
  \begin{equation}%
    \label{h-ups-2}
    \log |H| \leq 2 (\H(\phi(U_A)) + \H(\phi(U_B))) \leq 40k
  \end{equation}
  and
  \begin{equation}%
    \label{dist-proj-2}
    \H(\tilde\phi(U_A)) + \H(\tilde \phi(U_B)) \leq 8C d \bigl(1 + d^{C_{\PFR} - 1}\bigr)
  \end{equation}
  where 
  \begin{equation}%
    \label{d-def}
    d \coloneqq d_{\ent}\bigl(\tilde\phi(U_A), \tilde\phi(U_B)\bigr).
  \end{equation}
  Now by~\eqref{ent-here},~\eqref{h-ups} we also have
  \begin{equation}%
    \label{hab}
    \H(\tilde\phi(U_A)) + \H(\tilde \phi(U_B)) \leq 20k.
  \end{equation}

  In the following, set $\gamma \coloneqq 1/C_{\PFR}$ for convenience. If $d \geq 1$ then taking~\eqref{dist-proj-2} to the power $\gamma$ times~\eqref{hab} to the power $1 - \gamma$ gives
  \[
    \H(\tilde\phi(U_A)) + \H(\tilde \phi(U_B)) \leq 20 C k^{1 - \gamma} d.
  \]
  If $d \leq 1$ then the right-hand side of~\eqref{dist-proj-2} is $\leq 16C d$.
  Thus in all cases we have
  \begin{equation}%
    \label{habby}
    \H(\tilde\phi(U_A)) + \H(\tilde \phi(U_B)) \leq 20 C (1 + k^{1 - \gamma}) d.
  \end{equation}

  Now if $H$ is all of $\F_2^D$ then it follows from~\eqref{h-ups-2} (taking $C_2 = 40 / \log 2)$ that $D \leq C_2 k$, and so Theorem~\ref{thm83} is true simply by taking $A' = A$, $B' = B$. 

  Suppose, then, that $H$ is not all of $\F_2^D$.
  For $x, y \in \F_2^{D}/H$, denote by $A_x \coloneqq A \cap \tilde\phi^{-1}(x)$ and $B_y \coloneqq B \cap \tilde\phi^{-1}(y)$ the fibres of $A, B$ above $x, y$ respectively. Since we are assuming that $A, B$ do not sit inside cosets of a proper subgroup of $\Z^D$,  we may assume that at least one of $\tilde\phi(A), \tilde\phi(B)$ is not a singleton, and so
  \[
    |A_x| + |B_y| < |A| + |B|
  \]
  and $\H(\tilde\phi(U_A)) + \H(\tilde\phi(U_B)) > 0$, whereby $d>0$ by~\eqref{dist-proj-2}.
  Applying Lemma~\ref{lem:hom-lemma} once again, and noting~\eqref{d-def} and~\eqref{habby}, we find $x,y \in \F_2^D/H$ such that
  \begin{equation}%
    \label{eq:another-hom-bound}
    \log \frac{|A|}{|A_x|} + \frac{|B|}{|B_y|} 
    \leq 20 C (1 + k^{1-\gamma}) \bigl(k - d_\ent(U_{A_x}, U_{B_y}) \bigr)
  \end{equation}
  Set $k' = d_\ent(U_{A_x}, U_{B_y})$.
  By induction on $A_x$, $B_y$ we may find $A' \subseteq A_x$ and $B' \subseteq B_y$ such that $\dim A',\dim B' \leq C_2 k' \leq C_2 k$ and
  \[
    \log \frac{|A_x|}{|A'|} + \log \frac{|B_y|}{|B'|} \leq f(k').
  \]
  Adding this to~\eqref{eq:another-hom-bound} yields
  \begin{equation}\label{aaprime} \log \frac{|A|}{|A'|} + \log \frac{|B|}{|B'|} \leq f(k') + 20 C (1 + k^{1-\gamma}) (k - k') . \end{equation} 
  However, 
  \begin{align*}
  f(k') & = C_1 k' (1 + (k')^{1 - \gamma}) \\ & \leq C_1 k' (1 + k^{1 - \gamma}) \\ & = f(k) - C_1 (k - k')(1 + k^{1 - \gamma}).
   \end{align*}
  This, provided $C_1 \geq 20 C$, the right-hand side of~\eqref{aaprime} is at most $f(k)$, and this closes the induction. The proof is complete. 
\end{proof}

\appendix

\section{Basic facts about entropy}\label{basic-entropy-facts}

In this section we gather together basic facts about entropy, referring the reader to other sources (e.g.,~\cite[Appendix A]{tao-entropy} or~\cite{gray}) for the (standard, and mostly easy) proofs.

We begin with the most basic results. 

\subsubsection*{Basic entropy results}

If $X$ is an $S$-valued random variable for some finite $S$, the Shannon entropy is defined as
\[ \H(X) \coloneqq \sum_x p_X(x) \log \frac{1}{p_X(x)},\]
where $x$ is understood to range over $S$ and\footnote{We use the natural logarithm in this paper, but one could easily work with other bases of the logarithm if desired.} we adopt the convention that any term involving a factor of $p_X(x)$ vanishes when $p_X(x)=0$.  From Jensen's inequality we have
\begin{equation}\label{convexity-bound} \H(X) \leq \log |S|.\end{equation}
Also,
\[ \H(X) = \sum_x p_X(x) \log \frac{1}{p_X(x)} \geq \min_{x : p_X(x)>0} \log \frac{1}{p_X(x)},\] and therefore
\begin{equation}\label{px-lower} \max_x p_X(x) \geq e^{-\H(X)}.\end{equation}

If $X, Y$ are random variables then 
\begin{equation}\label{union} \H(X, Y) \leq \H(X) + \H(Y),\end{equation} 
and equality occurs if $X, Y$ are independent. At the other end of the spectrum, if $X$ determines $Y$ then $\H(X, Y) = \H(X)$. See for instance~\cite[Lemma 2.3.2]{gray}.

\subsubsection*{Conditional entropy}
We define
\[ \H(X | Y) = \sum_y p_Y(y) \H(X | Y = y).\] Then we have the \emph{chain rule}
\[ \H(X, Y) = \H(X | Y) + \H(Y).\] If $Y = f(X)$ for some function $f$ then, since $\H(X, Y) = \H(X)$, it follows that
\begin{equation}\label{h-ups} \H(f(X)) \leq \H(X).\end{equation}

\subsubsection*{Submodularity} For any three random variables $A, B, C$ we have the submodularity inequality
\begin{equation}\label{submodularity} \H(A, B, C) + \H(C) \leq \H(A, C) + \H(B,C)\end{equation}
(which is equivalent to the non-negativity of the conditional mutual information $\mathbf{I}(A:B|C)$); see for instance~\cite[Lemma 2.5.5]{gray}.

An equivalent and useful way to write the submodularity inequality is 
\begin{equation}\label{submodularity-2} \H(A | B, C) \leq \H(A | C).\end{equation}
Note also that, if $B$ determines $C$, then $\H(A, B, C) = \H(A, B)$ and $\H(B, C) = \H(B)$, and submodularity implies that 
\begin{equation}\label{submodularity-3} \H(A | B) \leq \H(A | C).\end{equation}

\subsubsection*{Kullback--Leibler Divergence.}
Suppose that $X, Y$ are random variables with distribution functions $\mu_X, \mu_Y$ respectively. Then we define 
\[ D_{\KL}(X \Vert Y) \coloneqq  \sum_t \mu_X(t) \log \left(\frac{\mu_X(t)}{\mu_Y(t)}\right).\]
It is conventional to define the summand here to be $0$ if $\mu_X(t) = 0$ and $\infty$ if $\mu_Y(t) = 0$ but $\mu_X(t) \neq 0$; in practice, we will avoid the latter situation.

It is convenient to relate this to the \emph{cross-entropy} 
\begin{equation}\label{cross-ent-def} \H( X : Y) \coloneqq  \sum_t \mu_X(t) \log \frac1{\mu_Y(t)}\end{equation} (where the same conventions are in force). Thus
\begin{equation}\label{kl-cross} D_{\KL}(X \Vert Y) = \H(X : Y) - \H(X).\end{equation}
In particular, if $X$ takes values in a finite set $S$, then $\H(X : U_S) = \log |S|$ and thus
\begin{equation}\label{dkus}
 D_\KL(X \Vert U_S) = \log |S| - \H(X).
\end{equation}
Note that $\H(X : Y)$ is \emph{not} at all the same thing as $\H(X, Y)$ (or $\H(X|Y)$). Indeed, the former depends only on the distribution functions of $X, Y$ and not in any way on their dependence, and it is also asymmetric in that in general $\H(X : Y) \neq \H(Y : X)$.
From a standard application of Jensen's inequality we obtain \emph{Gibbs' inequality}
\begin{equation}\label{gibbs} D_{\KL}(X \Vert Y) \geq 0\end{equation}
(see e.g.,~\cite[Theorem 2.3.1]{gray});  
we also have the well known \emph{Pinsker's inequality}
\begin{equation}\label{pinsker-eq} \sum_t |p_X(t) - p_Y(t)| \leq \sqrt{2 D_{\KL}(X \Vert Y)},\end{equation}
see e.g.,~\cite[Lemma 5.2.8]{gray}. 

Now we turn to some simple results about $G$-valued random variables, where $G$ is abelian, and we assume all random variables to have finite support. The reader may wish to recall the definitions of $d_{\ent}$ and $d^*_{\ent}$, given at~\eqref{dent-def} and~\eqref{max-dist} respectively.

First, if $X, Y$ are independent such variables then

\begin{equation}\label{simple-sumset} \H(X - Y) \geq \H(X-Y|Y) = \H(X).\end{equation}
From this we see that
\begin{equation}\label{ent-lower}
d_\ent(X,Y)=d_\ent(Y,X) \geq \frac{|\H(X)-\H(Y)|}{2} \geq 0.
\end{equation}

Let $X$ be a $G$-valued random variable, and let $H$ be a finite subgroup of $G$. Denote by $\pi : G \rightarrow G/H$ the quotient map.
Let $U_H$ be a uniform random variable on $H$, independent of $X$. Then we have
\begin{equation}\label{hxuh} \H(X + U_H) = \H(\pi(X)) + \H(U_H) = \H(\pi(X)) + \log |H|.\end{equation}
It follows that 
\begin{equation}\label{dist-proj} d_\ent(X, U_H) = \H(\pi(X)) + \frac{1}{2} (\log |H| - \H(X)).\end{equation}
From this and~\eqref{ent-lower} we have
\begin{equation}\label{ent-proj-dist} \H(\pi(X)) \leq 2 d_{\ent}(X, U_H).\end{equation}
Also, from Lemma~\ref{improv-ruzsa} and $d_\ent(U_H,U_H)=0$ we observe that
\[ d_\ent^*(X,U_H) =d_\ent(X,U_H).\]

Finally, if $X,Y,Z$ are independent $G$-valued random variables, we observe from the Gibbs inequality~\eqref{gibbs} the useful bound
\begin{align}\nonumber
  \H(Z - Y) & - \H(Y) \leq \H(Z - Y : X ) - \H(Y) \\ \nonumber
                    & = \sum_z p_Z(z) \bigl( \H(z - Y : X ) - \H(z - Y)\bigr)\\
                    & = \sum_z p_Z(z) D_{\KL}(z - Y \Vert X ) \label{hzy}
\end{align}
where we have used the permutation-invariance of Shannon entropy to observe that $\H(z-Y)=\H(Y)$, as well as the fact that $p_{Z-Y}(t) = \sum_z p_Z(z) p_{z - Y}(t)$.  Note that we in fact have equality when $X=Z-Y$.  

\section{Energy, entropy and doubling}%
\label{app:energy-entropy}

In this section we prove the inequalities~\eqref{structure-ineqs}. Recall the statement, which is that
\begin{equation}\label{appB-ineq-repeat}  \frac{|A|^3}{\mathrm{E}[A]} \leq \sigma_{\ent}[A] \leq \sigma[A].\end{equation}

\begin{proof}  Denote $X \coloneqq U_A+U'_A$ to be the sum of two independent uniform random variables on $A$. The right-hand inequality is immediate from the inequality $\H(X) \leq \log |A+A|$, which is a special case of Jensen's inequality.  As for the left-hand inequality, observe that
\[ p_X(x) \coloneqq  \frac{|A \cap (x - A)|}{|A|^2}.\] 
and then by the weighted AM--GM inequality, 
\[ e^{-\H(X)} = \prod_x p_X(x)^{p_X(x)} \leq \sum_x p_X(x)^2 = \frac{\mathrm{E}[A]}{|A|^4}.\] The result follows immediately.\end{proof}

The above argument can be reformulated in terms of the \emph{R\'enyi entropies} $\H_\alpha(X)$, defined for $\alpha \neq 1$ by
\[ \H_{\alpha}(X) \coloneqq  \frac{1}{1 - \alpha} \log \left(\sum_x p_X(x)^{\alpha}\right)\] 
and extended by continuity to $\alpha=1$ by setting $\H_1(X) \coloneqq \H(X)$.  A brief calculation reveals the identities
\begin{align*}
\exp(\H_0(X)) &=  |A + A|\\
\exp(\H_1(X)) &= \sigma_{\ent}[A] |A| \\
\exp(\H_2(X)) &=  \frac{|A|^4}{\mathrm{E}[A]},
\end{align*}
and the claim now follows from the well-known fact that the R\'enyi entropy $\H_\alpha(X)$ is non-increasing in $\alpha$.

We conclude with a simple example showing that both inequalities in~\eqref{appB-ineq-repeat} can be far from tight. Suppose that $n = 2m$ is even and $A = H \cup \{x_1,\dots, x_m\}$, with $H$ a subgroup of size $m$ and $x_1,\dots, x_m$ highly dissociated with respect to $H$, for instance with $x_i + x_j - x_k - x_l \in H$ only if $\{i,j\} = \{k,l\}$. Then we have $|A|^3/\mathrm{E}[A] = 8 + o(1)$ as $n \rightarrow \infty$. Turning to $\sigma_{\ent}[A]$, we of course have $\H(U_A) = \log n$. The variable $U_A + U'_A$ may be conditioned to subvariables which are, respectively, uniformly distributed on $H$, on the set $\bigcup_{i = 1}^m (x_i + H)$, and on the multiset $\bigcup_{i,j = 1}^m \{x_i + x_j\}$, with the conditioning probabilities being $\frac{1}{4}, \frac{1}{2},\frac{1}{4}$. One therefore computes that $\H(U_A + U'_A) = (\frac{7}{4} + o(1)) \log n$ and so $\sigma_{\ent}[A] = n^{3/4 + o(1)}$. Finally, $\sigma[A] = (\frac{3}{4} + o(1))n$.

\end{document}